\documentclass[reqno,10pt]{amsart}
\usepackage{amsmath,amsfonts,amssymb,amsxtra,latexsym,amscd,enumerate,amsthm,verbatim}
\usepackage{graphicx}
\usepackage{cancel}
\usepackage{esint}

\usepackage[margin=1.30in]{geometry}
\numberwithin{equation}{section}

\usepackage[utf8]{inputenc} 

\usepackage[T1]{fontenc}

\usepackage{mathtools} 
\usepackage{enumitem} 
\usepackage[english]{babel} 
\usepackage{mathabx}

\usepackage[usenames,dvipsnames]{xcolor} 
\definecolor{dkblue}{RGB}{1,31,91} 
\usepackage[colorlinks=true, pdfstartview=FitV, linkcolor=dkblue, citecolor=dkblue, urlcolor=dkblue]{hyperref}

\def\p{\partial}

\def\be{{\beta}}

\def\p{\partial}

\def\be{\begin{equation}}
\def\ee{\end{equation}}
\newtheorem{theorem}{Theorem}[section]
\newtheorem{lemma}[theorem]{Lemma}

\newtheorem{proposition}[theorem]{Proposition}

\newtheorem{remark}[theorem]{Remark}

\def\beq{\begin{equation}}
\def\eeq{\end{equation}}

\def\beq{\begin{equation}}
\def\eeq{\end{equation}}

\begin{document}

\title{Self-similar solutions for the Muskat equation}

\author{Eduardo Garc\'ia-Ju\'arez$^\dagger$}

\author{Javier G\'omez-Serrano$^\ddagger$}

\author{Huy Q. Nguyen$^*$}

\author{Beno\^it Pausader$^\mathsection$}

\address{$^{\dagger, \ddagger}$ Departament de Matemàtiques i Informàtica, Universitat de Barcelona, 
Gran Via de les Corts Catalanes 585, 08007 Barcelona, Spain.}
\email{\href{mailto:egarciajuarez@ub.edu}{$^\dagger$egarciajuarez@ub.edu}}\email{\href{mailto:jgomezserrano@ub.edu}{$^\ddagger$jgomezserrano@ub.edu}}

\address{$^{\ddagger,*,\mathsection}$ Department of Mathematics, Brown University, Kassar House, 151 Thayer Street, Providence, RI 02912, USA.
}
\email{\href{javier\_gomez\_serrano@brown.edu}{$^\ddagger$javier\_gomez\_serrano@brown.edu}}
\email{\href{huy\_quang\_nguyen@brown.edu}{$^*$huy\_quang\_nguyen@brown.edu}}
\email{\href{benoit\_pausader@brown.edu}{$^\mathsection$benoit\_pausader@brown.edu}}

\begin{abstract}
We show the existence of self-similar solutions for the Muskat equation. These solutions are parameterized by $0<s \ll 1$; they are exact corners of slope $s$ at $t=0$ and become smooth in $x$ for $t>0$.
\end{abstract}

\maketitle



\section{Introduction}

\subsection{The Muskat problem}
The Muskat problem describes the evolution of the free boundary  between  immiscible and incompressible fluids permeating a porous medium in a gravity field. Each fluid is assumed to have constant physical properties and their velocities are governed by Darcy's law,
\begin{equation*}
    \frac{\mu}{\kappa}u(z,t)=-\nabla p(z,t)-\rho g(0,1),\qquad z\in\mathbb{R}^2,\quad t\in\mathbb{R}_+,
\end{equation*}
where $p$, $u$ denote the pressure and velocity of the fluids, $\rho$, $\mu$ are their density and viscosity, $\kappa$ is the permeability constant of the medium, and $g$ the gravitational constant.
With or without surface tension effects, it has long been known that the problem can be reduced to an evolution equation for the free interface \cite{Cordoba-Gancedo:contour-dynamics-3d-porous-medium-2007,Escher-Simonett:solutions-muskat-surface-tension-1997,Siegel-Caflisch-Howison:global-existence-muskat-2004}.
The case of a two-fluid graph interface \begin{equation*}
    \Gamma(t)=\{(x,f(x,t)):x\in\mathbb{R}\}
\end{equation*} 
with only gravity effects admits a particularly compact form \cite{Cordoba-Gancedo:contour-dynamics-3d-porous-medium-2007}, now called the \textit{Muskat equation}:
\begin{equation}\label{OMuskat}
\begin{split}
\partial_tf=\frac{1}{\pi}\int_{\mathbb{R}}\frac{\partial_x \Delta_\alpha f}{1+(\Delta_\alpha f)^2}d\alpha,\qquad\Delta_\alpha f(x)=\frac{f(x)-f(x-\alpha)}{\alpha},\qquad x\in\mathbb{R}.
\end{split}
\end{equation}
Above, all physical constants have been normalized for notational simplicity.
The Muskat equation is well-posed locally in time for sufficiently smooth initial data, and globally in time if the initial interface is sufficiently flat \cite{Ambrose:well-posedness-hele-shaw-2004, Constantin-Cordoba-Gancedo-Strain:global-existence-muskat-2013, Cordoba-Cordoba-Gancedo:interface-heleshaw-muskat-2011,Cordoba-Gancedo:contour-dynamics-3d-porous-medium-2007,Cordoba-Gancedo:maximum-principle-muskat-2009,Siegel-Caflisch-Howison:global-existence-muskat-2004,Yi:Loca-classical-Muskat-1996,Yi:Global-classical-Muskat-2003}.
Most notably, an initially smooth interface can turn \cite{Castro-Cordoba-Fefferman-Gancedo-LopezFernandez:rayleigh-taylor-breakdown-2012} and later lose regularity in finite time \cite{Castro-Cordoba-Fefferman-Gancedo:breakdown-muskat-2013}. Furthermore, many other behaviors are possible, with interfaces that turn and then go back to the graph scenario \cite{Cordoba-GomezSerrano-Zlatos:stability-shifting-muskat-2015,Cordoba-GomezSerrano-Zlatos:stability-shifting-muskat-II-2017}. Thus, finding criteria for global existence became one of the main questions for the Muskat equation. 
Since equation \eqref{OMuskat} has a natural scaling given by
\begin{equation*}
    f(x,t)\rightarrow \lambda^{-1} f(\lambda x, \lambda t), \qquad \lambda>0,
\end{equation*}
these criteria are stated in terms of critical regularity, i.e.,  spaces that scale like $\dot{W}^{1,\infty}$. In this sense, having the product of the maximal and minimal slopes strictly less than $1$  is sufficient for global existence  \cite{Cameron:global-wellposedness-muskat-slope-less-1-2019}. See also \cite{Deng-Lei-Lin:2d-muskat-monotone-data-2017}. Medium-size initial data in critical spaces but with uniformly continuous slope guarantees global wellposedness \cite{Constantin-Cordoba-Gancedo-Strain:global-existence-muskat-2013}. If the initial data is sufficiently small in $\dot{H}^{\frac32}$, then the slope can be arbitrarily large \cite{Cordoba-Lazar:global-wellposedness-muskat-H32-2018} and even unbounded \cite{Alazard-Nguyen:Endpoint-Sobolev-Muskat-2020,Alazard-Nguyen:Muskat-non-lipschitz-2020}. The result \cite{Alazard-Nguyen:Endpoint-Sobolev-Muskat-2020} also shows local existence and uniqueness in $H^{\frac32}$. This is currently the best (lowest) regularity result in terms of the space of the initial data, which is a problem that has garnered a lot of attention recently (e.g. \cite{Alazard-Lazar:Paranlinearization-Muskat-2020,Alazard-Nguyen:Muskat-Critical-II-2021,Chen-Nguyen-Xu:Muskat-C1-2021,Cheng-GraneroBelinchon-Shkoller:well-posedness-h2-muskat2016,Constantin-Gancedo-Shvydkoy-Vicol:global-regularity-muskat-finite-slope-2017,Matioc:Muskat-2d-formulations-2019,Nguyen:Muskat-Besov-small-2021,Nguyen-Pausader:Paradifferential-Muskat-2020}).

\subsection{Main result} In this paper, we show the existence of self-similar solutions for the Muskat problem. These solutions correspond to the global-in-time evolution of initially exact corners, and thus they do not fit into the aforementioned  results\footnote{The results in \cite{Cameron:global-Muskat-3D-2020}  allow for merely medium size bounded slopes but require sublinear growth of the profile, while our solutions grow linearly in space.}.

We can rewrite \eqref{OMuskat} in terms of a closed system for the slope $h=\partial_x f$:
\begin{equation}\label{DMuskat}
\begin{split}
0=\partial_th-\frac{1}{\pi}\int_{\mathbb{R}}\frac{\Delta_\alpha\partial_xh}{1+(\fint_\alpha h)^2}d\alpha+\frac{2}{\pi}\int_{\mathbb{R}}(\Delta_\alpha h)^2\cdot\frac{\fint_\alpha h}{(1+(\fint_\alpha h)^2)^2}d\alpha,\qquad\fint_\alpha h(x)=\frac{1}{\alpha}\int_{x-\alpha}^x h(z)dz.
\end{split}
\end{equation}
Plugging the ansatz $h(x,t)=k(x/t)$ in \eqref{DMuskat}, we arrive at the equation
\begin{equation}\label{SSMuskat}
\begin{split}
0=Sk+\frac{1}{\pi}\int_{\mathbb{R}}\frac{\Delta_\alpha\partial_{y}k}{1+(\fint_\alpha k)^2}d\alpha-\frac{2}{\pi}\int_{\mathbb{R}}(\Delta_\alpha k)^2\cdot\frac{\fint_\alpha k}{(1+(\fint_\alpha k)^2)^2}d\alpha,\qquad S:=y\partial_y.
\end{split}
\end{equation}
for which we construct a local curve of solutions:
%
\begin{theorem}\label{SSTheorem}
There exists $s_\ast>0$ such that for all $|s|<s_\ast$, there exists a self-similar solution of \eqref{DMuskat} $h_s(x,t)=k_s(x/t)$ satisfying $\lim_{y \to +\infty} k_s(y) = s$. In addition, we have that
\begin{equation*}
\begin{split}
\Vert k_s(y)-s\frac{2}{\pi}\arctan(y)\Vert_{H^1}+|s|\Vert \partial_s k_s(y)-\frac{2}{\pi}\arctan(y)\Vert_{H^1}&\lesssim |s|^3.
\end{split}
\end{equation*}

\end{theorem}

\begin{remark}
$(i)$ In particular, we see that $k_s\in L^\infty\setminus\dot{H}^\frac{1}{2}(\mathbb{R})$. $(ii)$ In fact, we will compare $ k_s$ and $(2s/\pi)\arctan((1+s^2/3)y)$, but one can readily see that the difference between the two ans\"atze is $O(s^3)$. $(iii)$ Due to the symmetry we can find solutions with negative $s$ by setting $k_{-s}(y) = -k_{s}(y)$. From now on, we assume $s \geq 0$.
\end{remark}


Despite the numerous works on the Muskat equation and the mathematically equivalent vertical Hele-Shaw problem, self-similar solutions were only known for the simplified thin film Muskat \cite{Escher-Matioc-Matioc:Think-film-Muskat-2012,Gancedo-GraneroBelinchon-Stefano:Thin-film-Muskat-2020,Laurencot-Matioc:Self-similarity-thin-film-Muskat-2017}. See also \cite{Escher-Matioc:Muskat-parabolicity-fingering-2011} where the authors find traveling solutions for the Muskat problem with surface tension effects included. 

From our numerical results, it might look surprising that, no matter how big the slope is, the initial corner instantly smooths out, as opposed to the known ``waiting time'' phenomenon in the Hele-Shaw problem \cite{Abedin-Schwab:Hele-Shaw-dini-2020,Bazaliy-Vasylyeva:Two-phase-Hele-Shaw-corners-without-surface-tension-2014,Bazaliy-Vasylyeva:Two-phase-Hele-Shaw-surface-tension-corner-2011,ChangLara-Guillen-Schwab:Hele-Shaw-2019,Choi-Jerison-Kim:One-phase-Hele-Sahw-Lipschit-2007,Escher-Simonett:solutions-muskat-surface-tension-1997,Kim:uniqueness-existence-hele-shaw-stefan-2003,Kim:long-time-regularity-hele-shaw-2006,Kim:one-phase-hele-shaw-2006}. We must note however that those works correspond to a horizontal Hele-Shaw cell (hence without gravity) and some include fluid injection. 
Moreover, some of these works are in a one-phase setting, which even for Muskat significantly changes the possible behaviors \cite{Alazard:Convexity-Hele-Shaw-2021,Alazard-Meunier-Smets:Lyapunov-Hele-Shaw-2020,Castro-Cordoba-Fefferman-Gancedo:splash-singularities-one-phase-muskat-stable-2016,Dong-Gancedo-Nguyen:Global-one-phase-Muskat-2021,Gancedo-Strain:absence-splash-muskat-SQG-2014}.


\subsection{Outline of the paper}
The rest of the paper is structured as follows.  In Section \ref{notation} we summarize the notation that will be used along the paper. Next, in Section \ref{fixedpoint}, we first extract the quasilinear structure of \eqref{SSMuskat} and rewrite it as a fixed point equation. This section contains the proof of the main Theorem \ref{SSTheorem} via Proposition \ref{PropExistenceSS}. Section \ref{quantitative} contains the analysis of all the terms involved in the equation. We will use key cancellations provided by some ``elementary bricks'' that we will be able to extract through the symmetrization of the nonlinear terms. Finally in Section \ref{SectionNumerics} we illustrate our main Theorem by numerically computing part of the branch of self-similar solutions.

\section{Notations}\label{notation}

\subsection{General notations}

In the following, we fix $\varphi\in C^\infty_c(-4/3,4/3)$, a nonnegative even function such that $\varphi\equiv 1$ on $[-3/4,3/4]$. For simplicity of notation, we let $\fint_{\pm\alpha,0}f$ be an arbitrary function among
\begin{equation*}
\fint_{\pm\alpha,0}f\in\{f,\fint_{\alpha}f,\fint_{-\alpha}f\}.
\end{equation*}
We will work mostly on the Fourier side. We define the Fourier transform as
\begin{equation*}
\begin{split}
\mathcal{F}(f)(\xi)=\widehat{f}(\xi):=\frac{1}{\sqrt{2\pi}}\int_{\mathbb{R}} f(y)e^{-iy\xi}dy.
\end{split}
\end{equation*}
We note that the Fourier transform of a real odd function is an odd function taking purely imaginary values.
We define the Fourier multiplier $\vert\nabla\vert$ by
\begin{equation*}
\mathcal{F}\left\{\vert\nabla\vert f\right\}(\xi)=\vert\xi\vert\widehat{f}(\xi).
\end{equation*}
Given an operator $T$, we let
\begin{equation}\label{That}
\widehat{T}[f]=\mathcal{F}^{-1}T[\widehat{f}]
\end{equation}
be its conjugation by the Fourier transform. To avoid functional analytic considerations, we say that a functional $g\mapsto \mathcal{N}(g)$ is analytic {\it around} a function $L_s$ if for any choice of $g_1,g_2$ in the appropriate space (here $H^1$) the restricted function
\begin{equation*}
\begin{split}
N_{g_1,g_2}(t_1,t_2):=\mathcal{N}(L_s+t_1g_1+t_2g_2)
\end{split}
\end{equation*}
is analytic. In this case, we denote
\begin{equation}\label{ConventionLinMult}
\begin{split}
\mathcal{N}_1[g]:=\frac{d}{dt_1}N_{g,0}(0,0),\qquad 
\mathcal{N}_{\ge2}[g]:=\mathcal{N}(L_s+g)-\mathcal{N}(L_s)-\mathcal{N}_1[g],\qquad\mathcal{N}_{\ge1}[g]:=\mathcal{N}_1[g]+\mathcal{N}_{\ge2}[g],
\end{split}
\end{equation}
and we observe that
\begin{equation}\label{HOTTerms}
\begin{split}
\mathcal{N}_{\ge2}(L_s+g_2)-\mathcal{N}_{\ge2}(L_s+g_1)&=\int_{t_1=0}^1\frac{d}{dt_1}\frac{d}{dt_2}N_{g_1,g_2-g_1}(t_1,0)dt_1+\int_{\theta=0}^1(1-\theta)\frac{d^2}{d\theta^2}N_{g_1,g_2-g_1}(1,\theta)d\theta.
\end{split}
\end{equation}
In these notations, the ``center'' $L_s$ is implicit, but since we will always consider functionals around $L_s$ defined in \eqref{DefLs}, there should be no ambiguity.

\subsection{Algebra of operators}

We will use operators of the form
\begin{equation*}
\begin{split}
\mathcal{N}(f)(y)&:=\int_{\alpha=0}^\infty m(f;\alpha,y)\cdot G(f,\fint_\alpha f,\fint_{-\alpha}f)\frac{d\alpha}{\alpha}
\end{split}
\end{equation*}
associated to some multilinear function $f\mapsto m(f;\alpha,y)$ (i.e. multilinear in $f$ for each fixed $\alpha$, $y$) and some numerical analytic function $G$. For such operators, we compute that
\begin{equation}\label{N1}
\begin{split}
\mathcal{N}_1[g]&=\int_{\alpha=0}^\infty m_1(g,L_s,\dots,L_s;\alpha,y)\cdot G(L_s,\fint_\alpha L_s,\fint_{-\alpha}L_s)\frac{d\alpha}{\alpha}\\
&\quad+\int_{\alpha=0}^\infty m(L_s;\alpha,y)\cdot {\bf v}_g\cdot\nabla G(L_s,\fint_\alpha L_s,\fint_{-\alpha}L_s)\frac{d\alpha}{\alpha},\qquad{\bf v}_g:=(g,\fint_\alpha g,\fint_{-\alpha}g),
\end{split}
\end{equation}
and
\begin{equation}\label{N21}
\begin{split}
\frac{d}{dt_1}\frac{d}{dt_2}N_{g_1,g_2}(t_1,t_2)&=\int_{\alpha=0}^\infty m_2(g_1,g_2,L_{t_1,t_2};\alpha,y)\cdot G(L_{t_1,t_2},\fint_\alpha L_{t_1,t_2},\fint_{-\alpha}L_{t_1,t_2})\frac{d\alpha}{\alpha}\\
&\quad+\int_{\alpha=0}^\infty m_1(g_1,L_{t_1,t_2};\alpha,y)\cdot {\bf v}_{g_2}\cdot\nabla G(L_{t_1,t_2},\fint_\alpha L_{t_1,t_2},\fint_{-\alpha}L_{t_1,t_2})\frac{d\alpha}{\alpha}\\
&\quad+\int_{\alpha=0}^\infty m_1(g_2,L_{t_1,t_2};\alpha,y)\cdot {\bf v}_{g_1}\cdot\nabla G(L_{t_1,t_2},\fint_\alpha L_{t_1,t_2},\fint_{-\alpha}L_{t_1,t_2})\frac{d\alpha}{\alpha}\\
&\quad+\int_{\alpha=0}^\infty m(L_{t_1,t_2};\alpha,y)\cdot\nabla^2 G(L_{t_1,t_2},\fint_\alpha L_{t_1,t_2},\fint_{-\alpha}L_{t_1,t_2})[{\bf v}_{g_1},{\bf v}_{g_2}]\frac{d\alpha}{\alpha},
\end{split}
\end{equation}
where $m_1(f,g,\dots g)=d_gm_1\cdot f$ and similarly for $m_j$, and $L_{t_1,t_2}:=L_s+t_1g_1+t_2g_2$. Similarly
\begin{equation}\label{N22}
\begin{split}
\frac{d^2}{d\theta^2}N_{g_1,g_2-g_1}(1,\theta)&=\int_{\alpha=0}^\infty m_2(g_2-g_1,g_2-g_1,L_\theta;\alpha,y)\cdot G(L_\theta,\fint_\alpha L_\theta,\fint_{-\alpha}L_\theta)\frac{d\alpha}{\alpha}\\
&\quad+2\int_{\alpha=0}^\infty m_1(g_1-g_1,L_\theta;\alpha,y)\cdot {\bf v}_{g_2-g_1}\cdot\nabla G(L_\theta,\fint_\alpha L_\theta,\fint_{-\alpha}L_\theta)\frac{d\alpha}{\alpha}\\
&\quad+\int_{\alpha=0}^\infty m(L_\theta;\alpha,y)\cdot\nabla^2 G(L_\theta,\fint_\alpha L_\theta,\fint_{-\alpha}L_\theta)[{\bf v}_{g_2-g_1},{\bf v}_{g_2-g_1}]\frac{d\alpha}{\alpha}.
\end{split}
\end{equation}

\section{Reduction to a fixed point estimate}\label{fixedpoint}

\subsection{Analysis of the quasilinear structure}

We can extract the quasilinear part from \eqref{SSMuskat}. This will be defined in terms of two main terms. We define the function $F$ and the operator $W$ as follows
\begin{equation}\label{DefFV}
\begin{split}
F(t):=\left[1+t^2\right]^{-1},\qquad W[g](y):=\frac{1}{\pi}\int_{\mathbb{R}} \frac{(\fint_\alpha g(y))^2-g^2(y)}{1+(\fint_\alpha g(y))^2}\frac{d\alpha}{\alpha},
\end{split}
\end{equation}
and we obtain the following expression:
\begin{lemma}\label{LemNewEqSSP}

The self-similar profile $k$ satisfies
\begin{equation}\label{SSeq3}
\begin{split}
\vert\nabla\vert k-S\left[k+\frac{k^3}{3}\right]+W[k]\partial_yk&=R[k]+\mathcal{T}[k],
\end{split}
\end{equation}
where the semilinear terms are defined as
\begin{equation}\label{DefRT}
\begin{split}
R[h]&:=\frac{1}{\pi}\int_{\mathbb{R}}\partial_\alpha \left\{h(y-\alpha)\right\}\cdot \frac{(\fint_\alpha h)^2-h^2}{1+(\fint_\alpha h)^2}\frac{d\alpha}{\alpha},\\
\mathcal{T}[h]&:=(1+h^2)T[h];\qquad T[h]:=\frac{1}{\pi}\int_{\mathbb{R}}(\Delta_\alpha h)^2\cdot F^\prime(\fint_\alpha h)d\alpha.
\end{split}
\end{equation}

\end{lemma}

\begin{proof}[Proof of Lemma \ref{LemNewEqSSP}]

Dividing by $1+k^2$, it suffices to show that \eqref{SSMuskat} can be rewritten as
\begin{equation}\label{SSeq2}
\begin{split}
\left[F(k)\vert\nabla\vert-S\right]k+V[k]\partial_xk&=F(k)R[k]+T[k],\\
V[g](y)&:=\frac{1}{\pi}\int_{\mathbb{R}} \left\{F(g(y))-F(\fint_\alpha g(y))\right\}\frac{d\alpha}{\alpha}.
\end{split}
\end{equation}
The only nontrivial part is the decomposition of the second term in \eqref{SSMuskat}. We can expand
\begin{equation*}
\begin{split}
\frac{1}{\pi}\int_{\mathbb{R}}\partial_y(\Delta_\alpha g)\cdot F(\fint_\alpha g)d\alpha&=-F(g)\cdot\vert\nabla\vert g+\frac{1}{\pi}\int_{\mathbb{R}}\partial_y(\Delta_\alpha g)\cdot \left\{F(\fint_\alpha g)-F(g)\right\}d\alpha\\
&=-F(g)\cdot\vert\nabla\vert g+\partial_yg(y)\frac{1}{\pi}\int_{\mathbb{R}} \left\{F(\fint_\alpha g)-F(g)\right\}\frac{d\alpha}{\alpha}\\
&\quad-\frac{1}{\pi}\int_{\mathbb{R}}\partial_y g(y-\alpha)\cdot \left\{F(\fint_\alpha g)-F(g)\right\}\frac{d\alpha}{\alpha},
\end{split}
\end{equation*}
and rearranging the terms, we arrive at \eqref{SSeq2}.

\end{proof}

\subsection{Study of the linear equation: Duhamel formula}

Given a constant $\kappa>0$, we now consider the linear adjusted equation from \eqref{SSeq2}
\begin{equation}\label{LinAdEq}
\begin{split}
\left(\kappa\vert\nabla\vert-S\right)k&=p_1+\partial_yp_2
\end{split}
\end{equation}
for an odd function $k$. Taking the Fourier transform and using Duhamel's formula, we obtain the ODE
\begin{equation*}
\begin{split}
\partial_\xi\left\{\xi e^{\kappa\vert\xi\vert}\widehat{k}\right\}&=e^{\kappa\vert\xi\vert}\widehat{p}_1+i\xi e^{\kappa\vert\xi\vert}\widehat{p}_2.
\end{split}
\end{equation*}
If we assume that $\xi \widehat{k}(\xi)$ is continuous at the origin, and we integrate from $0$, we find that
\begin{equation}\label{SolDuhamel}
\begin{split}
\widehat{k}(\xi)&=C\frac{e^{-\kappa\vert\xi\vert}}{\xi}+{\bf T}_{-1}[\widehat{p}_1]+{\bf T}_0[\widehat{p}_2],
\end{split}
\end{equation}
for some constant $C$, where the linear operators are given by
\begin{equation}\label{SolvingOp1}
\begin{split}
{\bf T}_{-1}:g&\mapsto \frac{1}{\xi}\int_{\eta=0}^\xi e^{\kappa(\vert\eta\vert-\vert\xi\vert)}g(\eta)d\eta,\qquad
{\bf T}_0:g\mapsto \frac{i}{\xi}\int_{\eta=0}^\xi \eta e^{\kappa(\vert\eta\vert-\vert\xi\vert)}g(\eta)d\eta.
\end{split}
\end{equation}
In particular, under our assumptions, the odd solutions to the free equation $\left(\kappa\vert\nabla\vert-S\right)k=0$ with the condition $\lim_{y\to+\infty}k(y)=s$ are given by
\begin{equation}\label{FreeSolutions1}
\begin{split}
L_{s,\kappa}(y)&:=\frac{2s}{\pi}\arctan(y/\kappa).
\end{split}
\end{equation}
From now on, we shall restrict ourselves to the study of odd solutions.

\subsubsection{Operator estimates}

It remains to estimate the solution operators from \eqref{SolvingOp1}.

\begin{lemma}\label{TOperators}

The operators
\begin{equation*}
\begin{split}
{\bf T}_{-1}:g&\mapsto \frac{1}{\xi}\int_{\eta=0}^\xi e^{\kappa(\eta-\xi)}g(\eta)d\eta,\qquad
{\bf T}_0:g\mapsto \frac{i}{\xi}\int_{\eta=0}^\xi \eta e^{\kappa(\eta-\xi)}g(\eta)d\eta,\\
\end{split}
\end{equation*}
defined for functions on $L^2_{loc}(0,\infty)$ satisfy the boundedness properties
\begin{equation*}
\begin{split}
\Vert \xi {\bf T}_{-1}\Vert_{L^2\to L^2}+\Vert {\bf T}_{-1}\Vert_{L^3\to L^2}&\lesssim 1,\\
\Vert {\bf T}_0\Vert_{L^2\to L^2}+\Vert {\bf T}_0\Vert_{L^2(\xi^2 d\xi)\to L^2(\xi^2 d\xi)}&\lesssim 1.\\
\end{split}
\end{equation*}

\end{lemma}

\begin{proof}[Proof of Lemma \ref{TOperators}]

To control ${\bf T}_{-1}$, we first use H\"older's inequality to bound
\begin{equation*}
\begin{split}
\vert {\bf T}_{-1}[g](\xi)\vert&\le \xi^{-\frac{1}{p}}\left(\frac{1-e^{-\kappa\xi}}{\kappa\xi}\right)^\frac{1}{p^\prime}\left(\int_{\eta=0}^\xi e^{\kappa(\eta-\xi)}\vert g(\eta)\vert^pd\eta\right)^\frac{1}{p}\lesssim\Vert g\Vert_{L^p}\cdot\min\{\xi^{-\frac{1}{p}},\xi^{-1}\},
\end{split}
\end{equation*}
which shows the second bound. For the first, we compute by duality that
\begin{equation*}
\begin{split}
\langle \xi {\bf T}_{-1}[g],h\rangle&=\iint K(\xi,\eta)\overline{g}(\eta)h(\xi)d\xi d\eta,\qquad K(\xi,\eta):= \mathfrak{1}_{\{0\le\eta\le\xi\}}e^{\kappa(\eta-\xi)},
\end{split}
\end{equation*}
and Schur's test allows to conclude. This also gives the second estimate on ${\bf T}_0$, and the first one follows a similar proof with kernel $K^\prime(\xi,\eta)=(\eta/\xi)K(\xi,\eta)$.

\end{proof}

\subsection{Fixed point formulation}

\subsubsection{Linearization at self-similar profile}

The key observation we will use is that for the solutions of the linearized problem \eqref{FreeSolutions1}, the quasilinear part simplifies significantly since $L_{s,\kappa}^2$ is almost constant away from a neighborhood of the origin:
\begin{equation*}
\begin{split}
L_{s,\kappa}^2(y)=s^2+p(y),\qquad p(y)=O(\langle \kappa/y\rangle^{-1}),\qquad L_{s,\kappa}^3=s^2L_{s,\kappa}+(L_{s,\kappa}^2-s^2)L_{s,\kappa}.
\end{split}
\end{equation*}
We can now linearize \eqref{SSeq3} at the constant $s$ to get an equation of the form \eqref{LinAdEq} with
\begin{equation}\label{DispRel}
\kappa:=\left[1+s^2/3\right]^{-1}.
\end{equation}
More precisely, we obtain
\begin{equation*}
\begin{split}
\vert\nabla\vert k-\left[1+s^2/3\right]Sk=\frac{1}{3}S\left[(k^2-s^2)k\right]-W[k]\partial_yk+R[k]+\mathcal{T}[k].
\end{split}
\end{equation*}
which we prefer to rewrite via a normal form as
\begin{equation}\label{SSeq3Pert1}
\begin{split}
\left[\kappa\vert\nabla\vert -S\right]h&=\frac{\kappa^2}{3}\vert\nabla\vert \left[(k^2-s^2)k\right]-\kappa W[k]\partial_yk+\kappa R[k]+\kappa \mathcal{T}[k],\\
h&:=k+\frac{\kappa}{3}(k^2-s^2)k.
\end{split}
\end{equation}

\subsubsection{Fixed point formulation}

We can now seek solutions of \eqref{SSeq3Pert1} as perturbations of \eqref{FreeSolutions1} when $s,\kappa$ are related as in \eqref{DispRel}. We seek solutions of the form
\begin{equation}\label{DefLs}
\begin{split}
k(y)=L_s(y)+g(y),\qquad L_s(y):=\frac{2s}{\pi}\arctan(y/\kappa)=s\frac{2}{\pi}\arctan((1+s^2/3)y).
\end{split}
\end{equation}
We note for later use that $L_s$ is a smooth function of $s$ and
\begin{equation}\label{DerLs}
\begin{split}
\partial_sL_s&=\frac{2}{\pi}\arctan(y/\kappa)+\frac{4}{3\pi}\frac{s^2y}{1+(y/\kappa)^2},\\
\partial_s(L_s^2-s^2)&=2s\left(s^{-2}(L_s^2-s^2)+\frac{2}{\pi}\arctan(y/\kappa)\frac{4}{3\pi}\frac{s^2y}{1+(y/\kappa)^2}\right).
\end{split}
\end{equation}
We define $\pi_s$ by
\begin{equation*}
\pi_s:=(L_s^2-s^2)L_s.    
\end{equation*}
Then, plugging \eqref{DefLs} into \eqref{SSeq3Pert1}, we obtain
\begin{equation*}
\begin{split}
h&=L_s+(\kappa/3)\pi_s+(1+\kappa(L_s^2-s^2/3))g+\kappa L_sg^2+(\kappa/3) g^3,\\
\left[\kappa\vert\nabla\vert -S\right]h&=\kappa^2\vert\nabla\vert\left[(1/3)\pi_s+(L_s^2-s^2/3)g+L_sg^2+g^3/3\right]\\
&\quad-\kappa W[L_s]\partial_yL_s-\kappa\left(W[L_s+g]-W[L_s]\right)\partial_y\left(L_s+g\right)-\kappa \partial_y\left(W[L_s]g\right)+\kappa g\partial_yW[L_s]\\
&\quad+\kappa (R[L_s]+R_1[g]+ R_{\ge2}[g]+\mathcal{T}[L_s]+\mathcal{T}_1[g]+\mathcal{T}_{\ge2}[g]).
\end{split}
\end{equation*}
Thus, using \eqref{SolDuhamel} and the notation \eqref{That}, we obtain two equations for $h$ (after identifying the element in the kernel for the second equation from the limit at $\infty$):
\begin{equation*}
\begin{split}
h&=L_s+(\kappa/3)\pi_s+\left\{3(1+s^2)/(3+s^2)+\kappa(L_s^2-s^2)\right\}g+\kappa L_sg^2+(\kappa/3)g^3\\
&=L_s+\kappa \widehat{{\bf T}}_{-1}\left\{(\kappa/3)\vert\nabla\vert \pi_s-W[L_s]\partial_yL_s+R[L_s]+\mathcal{T}[L_s]\right\}\\
&\quad+\kappa \widehat{{\bf T}}_{-1}\{\kappa\vert\nabla\vert (L_s^2-s^2/3)g-\partial_y\left(W[L_s]g\right)+g\partial_yW[L_s]-W_1[g]\p_yL_s+R_1[g]+\mathcal{T}_1[g]\\
&\quad+\kappa \widehat{{\bf T}}_{-1}\{\kappa\vert\nabla\vert(L_sg^2+g^3/3)-W_{\ge2}[g]\p_yL_s- W^{hi}_{\ge1}[g]\partial_yg-\partial_y(gW^{lo}_{\ge1}[g])+g\partial_yW^{lo}[g]\\
&\quad\, +R_{\ge2}[g]+\mathcal{T}_{\ge2}[g]\},
\end{split} 
\end{equation*}
where $R_1,\mathcal{T}_1$, $R_{\ge2}$, $\mathcal{T}_{\ge2}$ follow the convention in \eqref{ConventionLinMult} and $W=W^{lo}+W^{hi}$ is a decomposition later introduced in \eqref{SymDecW}. Combining the two equations for $h$, we arrive at the fixed-point formulation:
\begin{equation}\label{SSFixPointEquation}
\begin{split}
\Pi&=Ag+\mathcal{N}(g),\\
\end{split}
\end{equation}
with forcing term
\begin{equation}\label{SSPi}
\begin{split}
\Pi&:=(\kappa/3)\pi_s-(\kappa^2/3)\widehat{{\bf T}}_{-1}[\vert\nabla\vert\pi_s]+\kappa \widehat{{\bf T}}_{-1}\left[W[L_s]\partial_yL_s-R[L_s]-\mathcal{T}[L_s]\right],\\
\end{split}
\end{equation}
linear operator
\begin{equation}\label{SSLinearOp}
\begin{split}
Ag&:=-(1+\frac{2s^2}{3+s^2}+\kappa (L_s^2-s^2))g+\kappa^2 \widehat{{\bf T}}_{-1}\left[\vert\nabla\vert (L_s^2-s^2/3)g\right]\\
&\quad+\kappa \widehat{{\bf T}}_{-1}\left[g\partial_yW[L_s]-W_1[g]\partial_yL_s+R_1[g]\right]-\kappa \widehat{{\bf T}}_0[W[L_s]g]+\kappa \widehat{{\bf T}}_{-1}[\mathcal{T}_1[g]],
\end{split}
\end{equation}
and nonlinearity
\begin{equation}\label{SSNonlinearity}
\begin{split}
\mathcal{N}(f)&:=-\kappa (L_s+g/3)g^2+\kappa^2 \widehat{{\bf T}}_{-1}\left[\vert\nabla\vert(L_sg^2+g^3/3)\right]\\
&\quad\,\,+\kappa \widehat{{\bf T}}_{-1}\left[-W_{\ge2}[g]\p_yL_s- W^{hi}_{\ge1}[g]\partial_yg-g\partial_yW^{lo}_{\ge1}[g]\right]-\kappa\widehat{{\bf T}}_0[gW^{lo}_{\ge1}[g]]\\
&\quad\,\,+\kappa \widehat{{\bf T}}_{-1}\left[ R_{\ge2}[g]\right]+\kappa \widehat{{\bf T}}_{-1}\left[\mathcal{T}_{\ge2}[g]\right],
\end{split}
\end{equation}
where we used that $\widehat{{\bf T}}_{-1}[\partial_yf]=\widehat{{\bf T}}_0[f]$. In view of this, Theorem \ref{SSTheorem} follows from the following existence result.

\begin{proposition}\label{PropExistenceSS}

There exists $s_\ast$ such that for all $0<s<s_\ast$, there exists exactly one solution ${\bf g}_s$ of \eqref{SSFixPointEquation} with \eqref{SSPi}-\eqref{SSNonlinearity} in a ball in $H^1$, $B_{H^1}(0,s_\ast)$. In addition the mapping $s\mapsto {\bf g}_s$ is $C^1$ in $H^1$ and
\begin{equation}\label{SmoothnessGs}
\begin{split}
\Vert\partial_s{\bf g}_s\Vert_{H^1}\lesssim s^2.
\end{split}
\end{equation}

\end{proposition}

This proposition will be an easy consequence of the following quantitative estimates proved in the next section.

\begin{lemma}\label{SSLemPi}

There holds that
\begin{equation*}
\begin{split}
\Vert \Pi\Vert_{H^1}&\lesssim s^3,\qquad\Vert \partial_s\Pi\Vert_{H^1}\lesssim s^2.
\end{split}
\end{equation*}

\end{lemma}

\begin{lemma}\label{SSLemA}

There holds that
\begin{equation*}
\Vert A+Id\Vert_{H^1\to H^1}\lesssim s^2,\qquad\Vert \partial_sA\Vert_{H^1\to H^1}\lesssim s.
\end{equation*}
In particular, there exists $s_1>0$ such that $A$ is invertible in $H^1$ for all $0\le s\le s_1$.
\end{lemma}

\begin{lemma}\label{SSLemN}

There holds that $\mathcal{N}(0)=0$ and whenever $\Vert g_1\Vert_{H^1}+\Vert g_2\Vert_{H^1}\le 1$, 
\begin{equation*}
\begin{split}
\Vert \mathcal{N}(g_1)-\mathcal{N}(g_2)\Vert_{H^1}&\lesssim \Vert g_1-g_2\Vert_{H^1}\cdot\left[s^2+\Vert g_1\Vert_{H^1}^2+\Vert g_2\Vert_{H^1}^2\right].
\end{split}
\end{equation*}
\end{lemma}

\begin{proof}[Proof of Proposition \ref{PropExistenceSS}]

We are now ready to prove Proposition \ref{PropExistenceSS} via a fixed point formulation. For $\varepsilon>0$, we consider $X:=\{g\in H^1:\Vert g\Vert_{H^1}\le\varepsilon\}$ and we want to show that \eqref{SSFixPointEquation} has a unique solution in $X$, provided that $0<s\le s_\ast$ is small enough. We define the mapping
\begin{equation*}
\begin{split}
\Phi:g\mapsto A^{-1}\left[\Pi-\mathcal{N}(g)\right],
\end{split}
\end{equation*}
which is well defined on $H^1$ since $A$ is invertible for $0< s\ll 1$ by Lemma \ref{SSLemA}. Using Lemma \ref{SSLemA} and Lemma \ref{SSLemN}, we see that $\Phi:X\to X$ provided $0<\varepsilon\le\varepsilon_\ast$ is small enough. Finally, decreasing $\varepsilon_\ast$ and using Lemma \ref{SSLemN} again, we see that $\Phi$ is a contraction on $X$. By the Banach fixed point theorem it has a unique fixed point in $X$.

In addition, we can study the smoothness of $s\mapsto {\bf g}_s$. Deriving \eqref{SSFixPointEquation}, we find that
\begin{equation*}
\begin{split}
(A+d_{g_s}\mathcal{N})[\partial_sg_s]=\partial_s\Pi_s-\left(\partial_sA\right)g_s,
\end{split}
\end{equation*}
and using Lemma \ref{SSLemPi}, Lemma \ref{SSLemA}, and Lemma \ref{SSLemN} again, we deduce that $s\mapsto {\bf g}_s$ is $C^1$ and we have \eqref{SmoothnessGs}.

\end{proof}

It remains to prove Lemma \ref{SSLemPi}, Lemma \ref{SSLemA} and Lemma \ref{SSLemN}. This will be done in the next section.

\section{Quantitative analysis}\label{quantitative}

\subsection{Analysis of the elementary bricks}\label{ElementaryBricks}

We need to control many terms. Fortunately, in many of them, the key cancellation is provided by simple ``elementary bricks'' which can be analyzed separately. We define the quadratic expressions for $\alpha\ge0$:
\begin{equation}\label{QuadraticBricks}
\begin{split}
\delta_\alpha[f](y)&:=(\fint_\alpha f(y))^2-(\fint_{-\alpha} f(y))^2,\qquad J_\alpha[h](y):=\left(\fint_\alpha h(y)\right)^2-2h^2(y)+\left(\fint_{-\alpha}h(y)\right)^2,\\
\end{split}
\end{equation}
and let $\delta_\alpha[f,g]$ and $J_\alpha[f,g]$ denote the bilinear expression obtained by polarization. We see that $\delta_\alpha[f]$ is odd if $f$ is odd, while $J_\alpha[f]$ is even if $f$ is either even or odd.

We start by estimating the key cancellation.

\begin{lemma}\label{ControlDelta2}

Given $\delta_\alpha[f,g]$ defined in \eqref{QuadraticBricks}, $y\ge0$ and $\alpha\ge 0$, there holds that
\begin{equation}\label{Delta2Ls}
\begin{split}
s^{-2}\vert\delta_\alpha[L_s](y)\vert&\lesssim \mathfrak{1}_{\{0\le\alpha\le y\}}\alpha\ln\langle\alpha\rangle\cdot \min\{y,y^{-2}\}+\mathfrak{1}_{\{0\le y\le\alpha\}}y\min\{\alpha,\alpha^{-1}\} ,\\
s^{-2}\vert\partial_y\delta_\alpha[L_s](y)\vert&\lesssim \mathfrak{1}_{\{0\le\alpha\le y\}}\alpha\ln\langle\alpha\rangle\cdot \langle y-\alpha\rangle^{-1}\langle y\rangle^{-2}+\mathfrak{1}_{\{0\le y\le\alpha\}}\min\{\alpha,\alpha^{-1} \},
\end{split}
\end{equation}
and
\begin{equation}\label{DeltaLsg}
\begin{split}
\vert \delta_\alpha[L_s,g]\vert&\lesssim s\Vert g\Vert_{H^1}\cdot \min\{\alpha^\frac{1}{2},\alpha^{-\frac{1}{2}}\},\\
\vert \delta_\alpha[g,h]\vert&\lesssim \Vert g\Vert_{H^1}\Vert h\Vert_{H^1}\min\{\alpha^\frac{1}{2},\alpha^{-1}\}.
\end{split}
\end{equation}
In addition,
\begin{equation}\label{Deltag1g2L^2}
\begin{split}
\vert \delta_\alpha[L_s,g]\vert&\lesssim s\int_{\{\vert t\vert \le\alpha\}}\left[\vert g^\prime(y+t)\vert+\langle \alpha\rangle^{-1}\vert g(y+t)\vert\right] dt,\\
\vert \delta_\alpha[g_1,g_2]\vert&\lesssim \Vert g_2\Vert_{H^1}\int_{\{\vert t\vert\le\alpha \}}\min\{\alpha^{-\frac{1}{2}},\alpha^{-\frac{3}{2}}\}\vert g_1(y+t)\vert dt,\\
&\quad+\Vert g_1\Vert_{H^1}\int_{\{\vert t\vert\le\alpha \}}\min\{\alpha^{-\frac{1}{2}},\alpha^{-\frac{3}{2}}\}\vert g_2(y+t)\vert dt
\end{split}
\end{equation}
and a derivative brings powers of $\alpha^{-1}$,
\begin{equation}\label{ddxDeltag1g2L^2}
\begin{split}
\vert\alpha\partial_y\delta_\alpha[L_s,g]\vert(y)&\lesssim s\Big(\vert g(y+\alpha)\vert+\vert g(y)\vert+\vert g(y-\alpha)\vert+\frac{1}{\alpha}\int_{t=-\alpha}^\alpha \vert g(y+t)\vert dt\Big).
\end{split}
\end{equation}

\end{lemma}

\begin{proof}[Proof of Lemma \ref{ControlDelta2}]

The key observation is that $2\delta_\alpha[f,g]=\delta[f,g]+\delta[g,f]$ where
\begin{equation}\label{DefDelta2}
\begin{split}
\delta[f,g]&:=-\frac{1}{\alpha}\int_{-\alpha}^\alpha f(y+t)dt\cdot\frac{1}{\alpha}\int_{t=0}^\alpha\left\{g(y+t)-g(y-t)\right\}dt.\\
\end{split}
\end{equation}
We see that the first term in $\delta$ vanishes to first order for odd $f$ if $\alpha\gg1$, while the second term contains a difference and vanishes to first order for $\alpha\ll 1$ when $g$ is smooth. The bounds in \eqref{Delta2Ls} follow directly from the formula \eqref{DefDelta2} and the bounds
\begin{equation}\label{ProofDeltaLs1}
\begin{split}
s^{-1}\left\vert \frac{1}{\alpha}\int_{-\alpha}^\alpha L_s(y+t)dt\right\vert &\lesssim y/(\langle y\rangle+\langle\alpha\rangle),\\
s^{-1}\left\vert \frac{1}{\alpha}\int_{t=0}^\alpha\left\{L_s(y+t)-L_s(y-t)\right\}dt\right\vert&\lesssim \mathfrak{1}_{\{0\le\alpha\le y\}}\alpha\ln\langle\alpha\rangle\cdot\langle y\rangle^{-2}+\mathfrak{1}_{\{0\le y\le\alpha\}}\min\{\alpha,1\},
\end{split}
\end{equation}
and
\begin{equation*}
\begin{split}
s^{-1}\left\vert\frac{1}{\alpha}\int_{-\alpha}^\alpha L_s^\prime(y+t)dt\right\vert&\lesssim \langle y\rangle^{-1}\langle y-\alpha\rangle^{-1}\mathfrak{1}_{\{0\le\alpha\le  y\}}+\langle\alpha\rangle^{-1}\mathfrak{1}_{\{0\le y\le \alpha\}},\\
s^{-1}\left\vert \frac{1}{\alpha}\int_{t=0}^\alpha\left\{L^\prime_s(y+t)-L^\prime_s(y-t)\right\}dt\right\vert&\lesssim \alpha\min\{y,\langle y-\alpha\rangle^{-1}\langle y\rangle^{-2}\}\mathfrak{1}_{\{0\le\alpha\le y\}}+\min\{\alpha^2,\alpha^{-1}\}\mathfrak{1}_{\{0\le  y\le\alpha\}}.
\end{split}
\end{equation*}
For general odd functions, we obtain that
\begin{equation}\label{CompIntgDelta}
\begin{split}
\left\vert\frac{1}{\alpha}\int_{-\alpha}^\alpha g(y+t)dt\right\vert&\lesssim \min\{\sqrt{ y+\alpha},\langle\alpha\rangle^{-\frac{1}{2}}\}\Vert g\Vert_{H^1},\\
\left\vert\frac{1}{\alpha}\int_{t=0}^\alpha\left\{g(y+t)-g(y-t)\right\}dt\right\vert&\lesssim \int_{\{\vert t\vert\le\alpha\}}\vert g^\prime(y+t)\vert dt,\\
\left\vert\frac{1}{\alpha}\int_{t=0}^\alpha\left\{g(y+t)-g(y-t)\right\}dt\right\vert&\lesssim\min\{\alpha^\frac{1}{2},\alpha^{-\frac{1}{2}}\}\Vert g\Vert_{H^1},
\end{split}
\end{equation}
from which we deduce \eqref{DeltaLsg} and \eqref{Deltag1g2L^2}. In addition, we observe that
\begin{equation*}
\begin{split}
\delta[\partial_xg_1,g_2]&=\frac{g_1(y-\alpha)-g_1(y+\alpha)}{\alpha}\cdot\frac{1}{\alpha}\int_{t=0}^\alpha\left\{g_2(y+t)-g_2(y-t)\right\} dt,\\
\delta[g_1,\partial_xg_2]&=-\frac{1}{\alpha}\int_{-\alpha}^\alpha g_1(y+t)dt\cdot\frac{g_2(y+\alpha)-2g_2(y)+g_2(y-\alpha)}{\alpha},\\
\end{split}
\end{equation*}
and using \eqref{ProofDeltaLs1}, we easily arrive at \eqref{ddxDeltag1g2L^2}.

\end{proof}

\begin{lemma}\label{PropJ}

Let $J_\alpha$ be defined as in \eqref{QuadraticBricks}. For $1\leq p\leq \infty$, there holds that
\begin{equation}\label{DFintaDa}
\begin{split}
\Vert \fint_\alpha h\Vert_{L^\infty_{\alpha}L^p_y}+\Vert \alpha\partial_\alpha\left(\fint_{\pm\alpha}h\right)\Vert_{L^\infty_{\alpha}L^p_y}&\lesssim \Vert h\Vert_{L^p},\\
\Vert J_\alpha[h]\Vert_{L^\infty}+\Vert \alpha\partial_\alpha J_\alpha [h]\Vert_{L^\infty}&\lesssim \Vert h\Vert_{L^\infty}^2,
\end{split}
\end{equation}
and for $L_s$ we have the more precise bound,
\begin{equation}\label{JLs}
\begin{split}
s^{-2}\vert J_\alpha[L_s](y)\vert&\lesssim \alpha\ln\langle \alpha\rangle\cdot \langle y\rangle^{-2}\mathfrak{1}_{\{0\le\alpha\le y\}}+\min\{\alpha^2,1\}\mathfrak{1}_{\{0\le  y \le \alpha\}}\lesssim\min\{\alpha,1\},
\end{split}
\end{equation}
while
\begin{equation}\label{JgL}
\begin{split}
s^{-1}\vert J_\alpha[h,L_s](y)\vert&\lesssim \int_{\vert t\vert\le\alpha}\left[\vert h(y+t)\vert+\vert h^\prime(y+t)\vert\right]dt,\\
\vert J_\alpha[h_1,h_2](y)\vert&\lesssim \sum_{\{a,b\}=\{1,2\}}\Vert h_a\Vert_{L^\infty}\int_{\vert t\vert\le\alpha}\left[\vert h_b(y+t)\vert+\vert h_b^\prime(y+t)\vert\right]dt.
\end{split}
\end{equation}

\end{lemma}

\begin{proof}[Proof of Lemma \ref{PropJ}]

For \eqref{DFintaDa}, the estimates in $\fint_\alpha h$ follow from direct computations. They directly imply the estimates on $J_\alpha$. To analyze $J_\alpha$, we can rewrite $J_\alpha[h]=J[h,h]$ where\footnote{Note that since $J$ is not symmetric, $J_\alpha[g,h]=J[g,h]+J[h,g]$.}
\begin{equation}\label{DefIAlpha}
\begin{split}
J[g,h]&=\frac{1}{2\alpha}\int_{t=0}^\alpha(g(y+t)-2g(y)+g(y-t))dt\cdot\frac{1}{\alpha}\int_{t=0}^\alpha(h(y+t)+2h(y)+h(y-t))dt\\
&\quad+\frac{1}{2}\left(\frac{1}{\alpha}\int_{t=0}^\alpha(g(y+t)-g(y-t))dt\right)\cdot\left(\frac{1}{\alpha}\int_{t=0}^\alpha(h(y+t)-h(y-t))dt\right).
\end{split}
\end{equation}
For \eqref{JLs}, the second term can be estimated using \eqref{ProofDeltaLs1}, while for the first term, we rewrite
\begin{equation}\label{SecondSymmetrizationJ}
\begin{split}
\frac{1}{\alpha}\int_{t=0}^\alpha(h(y+t)-2h(y)+h(y-t))dt
&=\frac{1}{\alpha}\int_{u=-\alpha}^\alpha (\alpha-\vert u\vert)^2h^{\prime\prime}(y+u)du.
\end{split}
\end{equation}
It remains to check \eqref{JgL}. For the linear term, this follows from \eqref{DefIAlpha} and \eqref{ProofDeltaLs1} for the second term, while for the first term, we use \eqref{SecondSymmetrizationJ} in case $g=L_s$ and
\begin{equation*}
\begin{split}
\left\vert \frac{1}{\alpha}\int_{t=0}^\alpha(h(y+t)-2h(y)+h(y-t))dt\right\vert&\le \int_{\{\vert t\vert\le\alpha\}}\vert h^\prime(y+t)\vert dt
\end{split}
\end{equation*}
else. For the quadratic estimate, we proceed similarly for the first term in \eqref{DefIAlpha} and we use \eqref{CompIntgDelta} for the second term.

\end{proof}

\subsection{Control on the nonlinear operators}

In this section, we will upgrade the bounds on the basic bricks in Section \ref{ElementaryBricks} to bounds on more complicated operators. Since they will not be multilinear, we will only operate under the assumption that
\begin{equation}\label{GlobalBoundG1G2G3}
\Vert g_1\Vert_{H^1}+\Vert g_2\Vert_{H^1}+\Vert g_3\Vert_{H^1}\le 1.
\end{equation}

\subsubsection{Symmetrization of the nonlinear operators}

We can symmetrize to obtain a key cancellation in the operators involved in \eqref{SSeq3}. The main observation is that one can always extract $\delta$ defined in \eqref{DefDelta2}. Using that
\begin{equation}\label{GSymRule}
\begin{split}
\int_{\mathbb{R}}H(\alpha)\frac{d\alpha}{\alpha}=\int_{\mathbb{R}_+}\left\{H(\alpha)-H(-\alpha)\right\}\frac{d\alpha}{\alpha},
\end{split}
\end{equation}
and the symmetrization formulas
\begin{equation}\label{SymF}
\begin{split}
F(a)-F(b)&=(b^2-a^2)\cdot F(a)F(b),\\
F(a)+F(b)&=(2+a^2+b^2)\cdot F(a)F(b),
\end{split}
\end{equation}
\begin{equation}\label{ProdSym}
\begin{split}
2(a_+b_+-a_-b_-)&=(a_+-a_-)(b_++b_-)+(a_++a_-)(b_+-b_-),\\
2(a_+b_++a_-b_-)&=(a_++a_-)(b_++b_-)+(a_+-a_-)(b_+-b_-),
\end{split}
\end{equation}
we can rewrite $W[g]=W[g,g;g]$, where
\begin{equation*}\label{SymW}
\begin{split}
W[g_1,g_2;g_3]&=\frac{1}{\pi}\int_{\mathbb{R}_+} \delta_\alpha[g_1,g_2]\cdot F(\fint_\alpha g_3)\cdot F(\fint_{-\alpha} g_3)\cdot (1+g_3^2)\frac{d\alpha}{\alpha}\\
&=\int_{\mathbb{R}_+} \delta_\alpha[g_1,g_2]\cdot G(g_3,\fint_\alpha g_3,\fint_{-\alpha} g_3)\frac{d\alpha}{\alpha}
\end{split}
\end{equation*}
for some function $G$ analytic in a neighborhood of $(0,0,0)$. In fact, to properly control the linear part, it will be convenient to decompose further $W=W^{lo}+W^{hi}$ where
\begin{equation}\label{SymDecW}
\begin{split}
W^{hi}[g_1,g_2;g_3]&=\int_{\mathbb{R}_+}\delta_\alpha [g_1,g_2]\cdot G(g_3,\fint_\alpha g_3,\fint_{-\alpha} g_3)\varphi(\alpha)\frac{d\alpha}{\alpha},\\
W^{lo}[g_1,g_2;g_3]&=\int_{\mathbb{R}_+} \delta_\alpha[g_1,g_2]\cdot G(g_3,\fint_\alpha g_3,\fint_{-\alpha} g_3)(1-\varphi(\alpha))\frac{d\alpha}{\alpha}.
\end{split}
\end{equation}

\subsubsection{Control on symmetrized operators}

\begin{lemma}\label{lemW}

There holds that for $\ast\in\{hi,lo\}$,
\begin{equation}\label{WLs}
\begin{split}
\Vert W^\ast[L_s]\Vert_{L^\infty}+\Vert \langle y\rangle\partial_yW^\ast[L_s]\Vert_{L^\infty}&\lesssim s^2,
\end{split}
\end{equation}
and
\begin{equation}\label{W1Lsg}
\begin{split}
\Vert W^\ast_1[g]\Vert_{L^\infty}+\Vert W^{hi}_1[g]\Vert_{L^2}+\Vert \partial_y(W^{lo}_1[g])\Vert_{ L^2}&\lesssim s\Vert g\Vert_{H^1}.
\end{split}
\end{equation}
Finally, we have the nonlinear estimates assuming \eqref{GlobalBoundG1G2G3}
\begin{equation}\label{W2}
\begin{split}
\Vert W^\ast_{\ge2}[g]\Vert_{L^2\cap L^\infty}&\lesssim \Vert g\Vert_{H^1}^2,\\
\Vert W^\ast_{\ge2}[g_1]-W^\ast_{\ge2}[g_2]\Vert_{L^\infty}&\lesssim (s+\Vert g_1\Vert_{H^1}+\Vert g_2\Vert_{H^1})\Vert g_1-g_2\Vert_{H^1}.
\end{split}
\end{equation}

\end{lemma}

\begin{proof}[Proof of Lemma \ref{lemW}]

We start from the formula \eqref{SymDecW}. For \eqref{WLs}, it suffices to show that
\begin{equation}\label{ControlDs}
\Vert \int_{\mathbb{R}_+}\vert \delta_\alpha[L_s]\vert\frac{d\alpha}{\alpha}\Vert_{L^\infty}+\Vert \int_{\mathbb{R}_+}\vert \langle y\rangle\partial_y\delta_\alpha[L_s]\vert\frac{d\alpha}{\alpha}\Vert_{L^\infty}\lesssim s^2,
\end{equation}
which follows from \eqref{Delta2Ls} and the bounds
\begin{equation*}
\Vert L_s\Vert_{L^\infty}+\Vert \langle y\rangle\partial_yL_s\Vert_{L^\infty}\le s.
\end{equation*}

For \eqref{W1Lsg}, inspecting \eqref{N1}, and using \eqref{DFintaDa} and \eqref{ControlDs}, we see that it suffices to show that
\begin{equation}\label{ControlDs1}
\begin{split}
\Vert \int_{\mathbb{R}_+}\vert\delta_\alpha[L_s,g]\vert \frac{d\alpha}{\alpha}\Vert_{L^\infty}+\Vert \int_{\mathbb{R}_+}\vert \delta_\alpha[L_s,g]\vert \varphi(\alpha)\frac{d\alpha}{\alpha}\Vert_{L^2}+\Vert \int_{\mathbb{R}_+}\vert\partial_y\delta_\alpha[L_s,g]\vert (1-\varphi(\alpha)) \frac{d\alpha}{\alpha}\Vert_{L^2}&\lesssim s\Vert g\Vert_{H^1},
\end{split}
\end{equation}
and
\begin{equation*}
\begin{split}
\Vert \int_{\mathbb{R}_+}\vert\delta_\alpha[L_s]\vert\cdot\vert \fint_{\pm\alpha} g\vert \varphi(\alpha)\frac{d\alpha}{\alpha}\Vert_{L^2}+\Vert \int_{\mathbb{R}_+}\vert\delta_\alpha[L_s]\vert\cdot\vert \partial_y\fint_{\pm\alpha} g\vert (1-\varphi(\alpha))\frac{d\alpha}{\alpha}\Vert_{L^2}\lesssim s\Vert g\Vert_{L^2}.
\end{split}
\end{equation*}
We start with the above bound. First, we see from \eqref{Delta2Ls} that $s^{-1}\vert\delta_\alpha[L_s]\vert\lesssim s \min\{1,\sqrt{\vert \alpha\vert}\}$ and we compute that
\begin{equation*}
\begin{split}
\Vert \int_{\mathbb{R}_+}\vert\delta_\alpha[L_s]\vert\cdot\vert \fint_{\pm\alpha} g\vert \varphi(\alpha)\frac{d\alpha}{\alpha}\Vert_{L^2}&\lesssim s\iint_{\{\vert t\vert\le\vert\alpha\vert\le 2\}} \Vert  g(y+t)\Vert_{L^2}\varphi(\alpha)\frac{d\alpha}{\alpha^\frac{3}{2}}dt\lesssim s\Vert g\Vert_{L^2},\\
\end{split}
\end{equation*}
and similarly,
\begin{equation*}
\begin{split}
 \Vert\int_{\mathbb{R}_+}\vert\delta_\alpha[L_s]\vert\cdot\vert \partial_y\fint_{\pm\alpha} g\vert (1-\varphi(\alpha))\frac{d\alpha}{\alpha}\Vert_{L^2}&\lesssim s\int_{\{\alpha\ge1/2\}}\left[\Vert g(y\mp\alpha)\Vert_{L^2_y} +\Vert g(y)\Vert_{L^2_y}\right]\frac{d\alpha}{\alpha^2}\lesssim s\Vert g\Vert_{L^2}.
 \end{split}
\end{equation*}
The first term in \eqref{ControlDs1} can be controlled through \eqref{DeltaLsg}, the second term can be controlled through \eqref{Deltag1g2L^2} and the last term through \eqref{ddxDeltag1g2L^2}.

Finally, we consider \eqref{W2}. Inspecting \eqref{HOTTerms}, \eqref{N21} and \eqref{N22}, we see that we need to show
\begin{equation}\label{ControlDs2}
\begin{split}
\Vert \int_{\mathbb{R}_+}\vert \delta_\alpha[g_1,g_2]\vert\frac{d\alpha}{\alpha}\Vert_{L^\infty\cap L^2}&\lesssim \Vert g_1\Vert_{H^1}\Vert g_2\Vert_{H^1},\\
\Vert \int_{\mathbb{R}_+}\vert \delta_\alpha[L_s,g_1]\cdot\fint_{\pm\alpha,0} g_2\vert\frac{d\alpha}{\alpha}\Vert_{L^\infty\cap L^2}&\lesssim s\Vert g_1\Vert_{H^1}\Vert g_2\Vert_{H^1},\\
\Vert \int_{\mathbb{R}_+}\vert \delta_\alpha[L_s]\cdot\fint_{\pm\alpha,0} g_1\cdot\fint_{\pm\alpha,0} g_2\vert\frac{d\alpha}{\alpha}\Vert_{L^\infty\cap L^2}&\lesssim s^2\Vert g_1\Vert_{H^1}\Vert g_2\Vert_{H^1}.
\end{split}
\end{equation}
The $L^\infty$ bounds follow from \eqref{DeltaLsg} and \eqref{ControlDs} together with the simple bound $\Vert g\Vert_{L^\infty}\lesssim \Vert g\Vert_{H^1}$. The $L^2$ bound follows from \eqref{Deltag1g2L^2} and \eqref{DeltaLsg} for the first two estimates while for the last, we see that
\begin{equation*}
\begin{split}
\int_{\mathbb{R}_+}\vert \delta_\alpha[L_s]\cdot\fint_{\pm\alpha,0} g_1\cdot\fint_{\pm\alpha,0} g_2\vert\frac{d\alpha}{\alpha}\lesssim I_1+I_2+I_3,\qquad I_j:=\iiint_{S_j} \vert\delta_\alpha[L_s](y)\vert \cdot\vert g_1(y+t)g_2(y+u)\vert \frac{d\alpha}{\alpha^3}dtdu,\\
S_1:=\{\vert u\vert\le\vert t\vert\le\alpha\le y\},\qquad
S_2:=\{\vert u\vert, y\le\vert t\vert\le\alpha\},\qquad
S_3:=\{\vert u\vert\le\vert t\vert\le y\le\alpha\},
\end{split}
\end{equation*}
and using \eqref{Delta2Ls}, we can compute that
\begin{equation*}
\begin{split}
\Vert I_1\Vert_{L^2}&\lesssim s^2\iint_{\{\vert u\vert\le\vert t\vert\}} \frac{\ln\langle t\rangle}{t\langle t\rangle^2} \cdot\Vert g_1(y+t)g_2(y+u)\Vert_{L^2_y} dtdu\lesssim s^2 \Vert g_1\Vert_{L^2}\Vert g_2\Vert_{L^\infty},\\
\end{split}
\end{equation*}
while
\begin{equation*}
\begin{split}
\vert I_2\vert&\lesssim s^2\iint \langle t\rangle^{-2}\vert g_1(y+t)g_2(y+u)\vert dtdu\lesssim s^2\Vert g_2\Vert_{L^2}\iint \langle t\rangle^{-\frac{3}{2}}\vert g_1(y+t)\vert dt,\\
\vert I_3\vert&\lesssim s^2\langle y\rangle^{-2}\iint_{\{\vert u\vert\le\vert t\vert\le y\}}\vert g_1(y+t)g_2(y+u)\vert dudt\lesssim s^2\langle y\rangle^{-1}\Vert g_1\Vert_{L^2}\Vert g_2\Vert_{L^2},
\end{split}
\end{equation*}
and we see that $I_2$ and $I_3$ are in $L^2$. This finishes the proof.

\end{proof}

We now turn to the first semilinear term.

\begin{lemma}\label{LemR}

Assume \eqref{GlobalBoundG1G2G3} and consider $R$ defined in \eqref{DefRT}. There holds that
\begin{equation}\label{RLs}
\begin{split}
\Vert R[L_s]\Vert_{L^\frac{4}{3}\cap L^2}&\lesssim s^3,\qquad
\Vert R_1[g]\Vert_{L^\frac{4}{3}\cap L^2}\lesssim s^2\Vert g\Vert_{H^1},\\
\Vert R_{\ge2}[g_1]-R_{\ge2}[g_2]\Vert_{L^1\cap L^2}&\lesssim \left[s^2+\Vert g_1\Vert_{H^1}^2+\Vert g_2\Vert_{H^1}^2\right]\Vert g_2-g_1\Vert_{H^1}.
\end{split}
\end{equation}

\end{lemma}

\begin{proof}[Proof of Lemma \ref{LemR}]

From \eqref{SymF} and \eqref{ProdSym} we deduce that
\begin{equation}\label{TripleSym}
\begin{split}
4(a_+b_+c_+-a_-b_-c_-)&=(a_+-a_-)(b_++b_-)(c_++c_-)+(a_++a_-)(b_+-b_-)(c_++c_-)\\
&\quad+(a_++a_-)(b_++b_-)(c_+-c_-)+(a_+-a_-)(b_+-b_-)(c_+-c_-).
\end{split}
\end{equation}
We can use \eqref{GSymRule}, \eqref{TripleSym} and \eqref{SymF} to decompose
\begin{equation*}
\begin{split}
R[h]&=-\frac{1}{\pi}\int_{\mathbb{R}}h^\prime(y-\alpha)\cdot \frac{(\fint_\alpha h)^2-h^2}{1+(\fint_\alpha h)^2}\frac{d\alpha}{\alpha}=\frac{1}{4\pi}(R^1-2R^2-R^3),
\end{split}
\end{equation*}
where $R^j=R^j[h;h]$ is given by
\begin{equation*}
\begin{split}
R^1[h_1;h_2]&=\int_{\mathbb{R}_+}\partial_\alpha\left\{h_1(y+\alpha)+h_1(y-\alpha)\right\}\cdot J_\alpha[h_2]\cdot(F_\alpha +F_{-\alpha})\frac{d\alpha}{\alpha},\\
R^2[h_1;h_2]&=\int_{\mathbb{R}_+}\partial_\alpha \left\{h_1(y+\alpha)-h_1(y-\alpha)\right\}\cdot\delta_\alpha[h_2](y) \cdot(1+h^2_2) \left(F_\alpha F_{-\alpha}\right)\frac{d\alpha}{\alpha},\\
R^3[h_1;h_2]&=\int_{\mathbb{R}_+}\partial_\alpha\left\{h_1(y+\alpha)+h_1(y-\alpha)\right\}\cdot \left[\delta_\alpha[h_2](y)\right]^2\cdot (F_\alpha F_{-\alpha})\frac{d\alpha}{\alpha},\\
\end{split}
\end{equation*}
with the notation
\begin{equation*}
\begin{split}
F_\alpha:=F(\fint_\alpha h_2).
\end{split}
\end{equation*}
From now on, we will denote $F_\alpha=F(\fint_{\pm\alpha} h)$ without distinction on $h$.

We start with the first estimate in \eqref{RLs}. First,  using \eqref{Delta2Ls} we observe that
\begin{equation}\label{R23Ls}
\begin{split}
\vert R^2[L_s;L_s]\vert+\vert R^3[L_s;L_s]\vert&\lesssim s\int_{\mathbb{R}_+}\langle y-\alpha\rangle^{-2}\vert \delta_\alpha[L_s]\vert\frac{d\alpha}{\alpha}\lesssim s^3\langle y\rangle^{-1}.
\end{split}
\end{equation}
In addition, using that
\begin{equation*}
\begin{split}
L_s^\prime(y-\alpha)-L_s^\prime(y+\alpha)&=\frac{2s\kappa}{\pi}\frac{4y\alpha}{(\kappa^2+(y-\alpha)^2)(\kappa^2+(y+\alpha)^2)},
\end{split}
\end{equation*}
and \eqref{JLs}, we see that
\begin{equation}\label{R1Ls}
\begin{split}
\vert R^1[L_s;h](y)\vert&\lesssim s\Vert h\Vert_{L^\infty}^2y\int \frac{d\alpha}{(\kappa^2+(y-\alpha)^2)(\kappa^2+(y+\alpha)^2)}\lesssim s\Vert h\Vert_{L^\infty}^2\cdot y\langle y\rangle^{-2}.
\end{split}
\end{equation}

\medskip

We now turn to the other estimates in \eqref{RLs} and we start with $\widetilde{R}^j[L_s+g]:=R^j[L_s;L_s+g]$. Inspecting \eqref{N1}, we see that the linear component follows from \eqref{R23Ls}, \eqref{R1Ls} and from the bound
\begin{equation*}
\begin{split}
s\int_{\mathbb{R}_+}\langle y-\alpha\rangle^{-2}\vert \delta_\alpha[L_s,g]\vert\frac{d\alpha}{\alpha}\lesssim s^2\Vert g\Vert_{H^1}\cdot \langle y\rangle^{-3/2},
\end{split}
\end{equation*}
which follows from \eqref{DeltaLsg}. Similarly, inspecting \eqref{HOTTerms} and \eqref{N21}-\eqref{N22}, we see that the higher order terms can be controlled similarly since we can easily estimate
\begin{equation*}
\begin{split}
s\int_{\mathbb{R}_+}\langle y-\alpha\rangle^{-2}\vert \delta_\alpha[g_1,g_2]\vert\frac{d\alpha}{\alpha}\lesssim s\Vert g_1\Vert_{H^1}\Vert g_2\Vert_{H^1}\cdot \langle y\rangle^{-3/2}
\end{split}
\end{equation*}
using \eqref{DeltaLsg} again.

\medskip

We now consider $R^j[g_1;L_s]$. Using Cauchy-Schwarz and \eqref{Delta2Ls}, \eqref{JLs}, we see that
\begin{equation*}
\begin{split}
R^{1,a}[h]&:=\int_{\mathbb{R}^+} \partial_y\left\{h(y+\alpha)-h(y-\alpha)\right\}\cdot J_\alpha[L_s]\cdot (F_\alpha +F_{-\alpha})\cdot \varphi(4\langle y\rangle^{-1}\alpha) \frac{d\alpha}{\alpha},\\
R^{2,a}[h]&:=\int_{\mathbb{R}^+} \partial_y\left\{h(y+\alpha)+h(y-\alpha)\right\}\cdot \delta_\alpha[L_s]\cdot (1+L_s^2)(F_\alpha F_{-\alpha})\cdot \varphi(4\langle y\rangle^{-1}\alpha) \frac{d\alpha}{\alpha},\\
\end{split}
\end{equation*}
satisfy
\begin{equation}\label{RjaL^2}
\begin{split}
\vert R^{1,a}[h]\vert+\vert R^{2,a}[h]\vert &\lesssim s^2\Vert \partial_yh\Vert_{L^2}\cdot \langle y\rangle^{-\frac{3}{2}}\ln\langle y\rangle,
\end{split}
\end{equation}
which gives an acceptable contribution. On the other hand, integrating by parts and using \eqref{DFintaDa}, we see that
\begin{equation*}
\begin{split}
R^{1,b}&=\int_{\mathbb{R}_+}\partial_\alpha\left\{h(y+\alpha)+h(y-\alpha)\right\}\cdot J_\alpha[L_s]\cdot(F_\alpha +F_{-\alpha})\cdot(1-\varphi(4\langle y\rangle^{-1}\alpha))\frac{d\alpha}{\alpha}\\
&=-\int_{\mathbb{R}_+}\left\{h(y+\alpha)+h(y-\alpha)\right\}\cdot \alpha\partial_\alpha\left\{J_\alpha[L_s]\cdot(F_\alpha +F_{-\alpha})\cdot\frac{1-\varphi(4\langle y\rangle^{-1}\alpha)}{\alpha}\right\} \frac{d\alpha}{\alpha}\\
\end{split}
\end{equation*}
satisfies
\begin{equation*}
\begin{split}
\vert R^{1,b}(y)\vert&\lesssim s^2\Vert h\Vert_{L^2}\langle y\rangle^{-\frac{3}{2}}
\end{split}
\end{equation*}
and similarly for $R^{2,b}$. $R^3$ can be treated similarly as $R^2$. To finish the proof, it only remains to show
\begin{equation*}
\begin{split}
\Vert R^j[g_1;L_s\!+\!g_1]\!-\!R^j[g_1;L_s]\!-\!(R^j[g_2;L_s\!+\!g_2]\!-\! R^j[g_2;L_s])\Vert_{L^1\cap L^2}&\lesssim \left[s^2\!+\!\Vert g_1\Vert_{H^1}^2\!+\!\Vert g_2\Vert_{H^1}^2\right]\Vert g_1\!-\!g_2\Vert_{H^1}.
\end{split}
\end{equation*}

We start with the case $j=2$. The case $j=3$ is similar and will not be detailed. We decompose
\begin{equation*}
\begin{split}
R^2[g_1;L_s+g_2]&=R^{2,lo}[g_1;L_s+g_2]+R^{2,hi}[g_1;L_s+g_2],\\
R^{2,hi}[g_1;h]&:=\int_{\mathbb{R}_+}\partial_\alpha \left\{g_1(y+\alpha)-g_1(y-\alpha)\right\}\cdot\delta[h,h](y) \cdot(1+h^2) \left(F_\alpha F_{-\alpha}\right)\cdot\varphi(\alpha)\frac{d\alpha}{\alpha},\\
R^{2,lo}[g_1;h]&:=\int_{\mathbb{R}_+}\left\{g_1(y+\alpha)-g_1(y-\alpha)\right\}\cdot\alpha\partial_\alpha \left\{\delta[h,h](y) \cdot(1+h^2) \left(F_\alpha F_{-\alpha}\right)\cdot\frac{1-\varphi(\alpha)}{\alpha}\right\}\frac{d\alpha}{\alpha}.
\end{split}
\end{equation*}
Inspecting \eqref{HOTTerms} and \eqref{N21}-\eqref{N22},  we see that to control $R^{2,hi}$, it suffices to show that
\begin{equation*}
\begin{split}
\Vert \int_{\mathbb{R}_+} \vert g_1^\prime(y\pm \alpha)\vert\cdot\vert \delta_{\alpha}[L_s](y) \vert \varphi(\alpha)\frac{d\alpha}{\alpha}\Vert_{L^1\cap L^2}&\lesssim s^2\Vert g_1^\prime\Vert_{L^2},\\
\Vert \int_{\mathbb{R}_+} \vert g_1^\prime(y\pm \alpha)\vert\cdot\vert \delta[g_2,L_s](y) \vert \varphi(\alpha)\frac{d\alpha}{\alpha}\Vert_{L^1\cap L^2}&\lesssim s\Vert g_1^\prime\Vert_{L^2}\Vert g_2\Vert_{H^1},\\
\Vert \int_{\mathbb{R}_+} \vert g_1^\prime(y\pm \alpha)\vert\cdot\vert \delta[g_2,g_3](y) \vert \varphi(\alpha)\frac{d\alpha}{\alpha}\Vert_{L^1\cap L^2}&\lesssim \Vert g_1^\prime\Vert_{L^2}\Vert g_2\Vert_{H^1}\Vert g_3\Vert_{H^1}.
\end{split}
\end{equation*}
The first estimate follows by Cauchy-Schwarz as in \eqref{RjaL^2}; the last two follow using \eqref{DeltaLsg} and \eqref{Deltag1g2L^2}. Independently,
\begin{equation*}
\begin{split}
\vert R^{2,lo}_{\ge1}[g_1;L_s+g_2]\vert &\lesssim \sum_\pm\left[s+\Vert g_2\Vert_{L^\infty}\right]\int_{\mathbb{R}_+}\vert g_1(y\pm\alpha)\cdot\left[\vert\fint_{\pm\alpha,0} g_2(y)\vert+\vert\alpha\partial_\alpha\fint_{\pm\alpha} g_2(y)\vert\right] \cdot\frac{1-\varphi(\alpha)}{\alpha^2}d\alpha,
\end{split}
\end{equation*}
so that this term can easily be handled using \eqref{DFintaDa}.

It remains to consider $R^1[g_1;L_s+g_2]$. The broad ideas are the same as for $R^2$, but we need to be slightly more careful because $J_\alpha$ has less good properties than $\delta_\alpha$ and we need to take advantage of the difference of derivative in $g_1$ to compensate for this. 

We claim that it suffices to show that for any dyadic number $A>0$, there exists a decomposition $R^1[g_1;L_s+g_2]=R^{1,hi}_A[g_1;L_s+g_2]+R^{1,lo}_A[g_1;L_s+g_2]$ such that
\begin{equation}\label{SuffR1}
\begin{split}
\Vert R^{1,\ast}[g;L_s+h]\Vert_{L^1\cap L^2}&\lesssim C(\ast,A,g)\left[s^2+\Vert h\Vert_{H^1}^2\right],\\
\Vert R^{1,\ast}[g;L_s+h_2]-R^{1,\ast}[g;L_s+h_1]\Vert_{L^1\cap L^2}&\lesssim C(\ast,A,g)\Vert h_2-h_1\Vert_{H^1}\left[s+\Vert h_1\Vert_{H^1}+\Vert h_2\Vert_{H^1}\right],\\
\end{split}
\end{equation}
where
\begin{equation}\label{SuffR12}
\begin{split}
C(hi,A,g):=A^{-1}\langle A\rangle^{-\frac{1}{2}}\Vert g^{\prime\prime}\Vert_{L^2},\qquad C(lo,A,g):=A\Vert g\Vert_{L^2}. 
\end{split}
\end{equation}
Indeed, to obtain \eqref{RLs}, we only need to replace $C(\ast,A,g)$ by $\Vert g\Vert_{H^1}$. To do this, we decompose $g=\sum_Bg_B$ using a Littlewood-Paley projection such that
\begin{equation*}
\begin{split}
\Vert g_B\Vert_{L^2}\lesssim \min\{1,B^{-1}\}\Vert g\Vert_{H^1},\qquad\Vert g_B^{\prime\prime}\Vert_{L^2}\lesssim\min\{B^2,B\}\Vert g\Vert_{H^1},
\end{split}
\end{equation*}
and for each $g_B$, we choose $A:=B^\delta+B^{1-\delta}$  for $0<\delta<1/5$ and we compute that
\begin{equation*}
\begin{split}
\sum_B \left\{C(lo,A,g_B)+C(hi,A,g_B)\right\}\lesssim\sum_B\min\{B^\delta,B^{-\delta}\}\Vert g\Vert_{H^1}\lesssim \Vert g\Vert_{H^1}.
\end{split}
\end{equation*}

It now remains to prove \eqref{SuffR1}-\eqref{SuffR12}. We decompose, for $A>0$ dyadic, $R^1[g_1;L_s+g_2]=R^{1,hi}_{A}[g_1;L_s+g_2]+R^{1,lo}_{A}[g_1;L_s+g_2]$ where
\begin{equation}\label{DefR1hilo}
\begin{split}
R^{1,hi}_A[g_1,h]&:=\int_{\mathbb{R}_+}\left\{g_1^\prime(y+\alpha)-g_1^\prime(y-\alpha)\right\}\cdot J_\alpha[h]\cdot(F_\alpha +F_{-\alpha})\cdot\varphi(A\alpha)\frac{d\alpha}{\alpha},\\
R^{1,lo}_A[g_1,h]&:=\int_{\mathbb{R}_+}\left\{g_1(y+\alpha)+g_1(y-\alpha)\right\}\cdot \alpha\partial_\alpha\left\{J_\alpha[h]\cdot(F_\alpha +F_{-\alpha})\cdot\frac{1-\varphi(A\alpha)}{\alpha}\right\}\frac{d\alpha}{\alpha}.
\end{split}
\end{equation}
For $R^{1,hi}_{A}$, using \eqref{HOTTerms} and \eqref{N21}-\eqref{N22}, it suffices to show that
\begin{equation*}
\begin{split}
\Vert \int_{\mathbb{R}_+}\int_{\{\vert t\vert\le\alpha\}}\vert g_1^{\prime\prime}(y+t)\vert\cdot \vert J_\alpha[L_s]\vert \cdot\vert \fint_{\pm\alpha,0}g_2(y)\vert \cdot\varphi(A\alpha)\frac{d\alpha}{\alpha}\Vert_{L^1\cap L^2}&\lesssim s^2A^{-1}(1+A)^{-1}\Vert g_1^{\prime\prime}\Vert_{L^2}\Vert g_2\Vert_{H^1},\\ 
\Vert \int_{\mathbb{R}_+}\int_{\{\vert t\vert\le\alpha\}}\vert g_1^{\prime\prime}(y+t)\vert\cdot \vert J_\alpha[L_s,g_2]\vert \cdot\varphi(A\alpha)\frac{d\alpha}{\alpha}\Vert_{L^1\cap L^2}&\lesssim sA^{-1}(1+A)^{-\frac{1}{2}}\Vert g_1^{\prime\prime}\Vert_{L^2}\Vert g_2\Vert_{H^1},\\ 
\Vert \int_{\mathbb{R}_+}\int_{\{\vert t\vert\le\alpha\}}\vert g_1^{\prime\prime}(y+t)\vert\cdot \vert J_\alpha[g_2,g_3]\vert \cdot\varphi(A\alpha)\frac{d\alpha}{\alpha}\Vert_{L^1\cap L^2}&\lesssim A^{-1}(1+A)^{-\frac{1}{2}}\Vert g_1^{\prime\prime}\Vert_{L^2}\Vert g_2\Vert_{H^1}\Vert g_3\Vert_{H^1}.
\end{split}
\end{equation*}
The first estimate follows from \eqref{DFintaDa} and the bound \eqref{JLs}. The second and third follow from \eqref{DFintaDa} and \eqref{JgL}. We now consider the contribution of $R^{1,lo}_{A}$. Inspecting \eqref{DefR1hilo}, it suffices to show that
\begin{equation*}
\begin{split}
\int_{\mathbb{R}_+} \Vert  g_1(y\pm\alpha) \fint_{\pm\alpha}g_2(y)\left[ \vert J_\alpha[L_s]\vert+\vert \alpha\partial_\alpha (J_\alpha[L_s])\vert\right]\Vert_{L^2\cap L^1}\cdot(1-\varphi(A\alpha))\frac{d\alpha}{\alpha^2}&\lesssim s^2A\Vert g_1\Vert_{L^2}\Vert g_2\Vert_{H^1},\\
\int_{\mathbb{R}_+} \Vert  g_1(y\pm\alpha)\cdot\left(\alpha\partial_\alpha \fint_{\pm\alpha}g_2(y)\right)\vert J_\alpha[L_s]\vert\Vert_{L^2\cap L^1}\cdot(1-\varphi(A\alpha))\frac{d\alpha}{\alpha^2}&\lesssim s^2A\Vert g_1\Vert_{L^2}\Vert g_2\Vert_{H^1},\\
\int_{\mathbb{R}_+}\Vert  g_1(y\pm\alpha) \left[\vert J_\alpha[g_2,g_3]\vert+\vert\alpha\partial_\alpha (J_\alpha[g_2,g_3])\vert\right] \Vert_{L^2\cap L^1}\cdot(1-\varphi(A\alpha))\frac{d\alpha}{\alpha^2}&\lesssim A\Vert g_1\Vert_{L^2}\Vert g_2\Vert_{H^1}\Vert g_3\Vert_{L^\infty},
\end{split}
\end{equation*}
where in the last estimate we consider $g_3=L_s$ or $g_3\in H^1$. These estimates all follow from \eqref{DFintaDa} and direct integration.

\end{proof}

\begin{lemma}\label{LemT}
Assume \eqref{GlobalBoundG1G2G3} and recall $T$ and $\mathcal{T}$ defined in \eqref{DefRT}. There holds that
\begin{equation}\label{BoundsTLs}
\begin{split}
\Vert T[L_s]\Vert_{L^\frac{4}{3}\cap L^2}&\lesssim s^3,\qquad \Vert T_1[g]\Vert_{L^\frac{4}{3}\cap L^2}\lesssim s^2\Vert g\Vert_{H^1}, 
\end{split}
\end{equation}
and
\begin{equation}\label{BoundsTLipsitz}
\begin{split}
\Vert T_{\ge2}[g_2]-T_{\ge2}[g_1]\Vert_{L^\frac{4}{3}\cap L^2}&\lesssim \left[s^2+\Vert g_1\Vert_{H^1}^2+\Vert g_2\Vert_{H^1}^2\right]\Vert g_2-g_1\Vert_{H^1}.
\end{split}
\end{equation}
Besides the same bound hold if we replace $T[h]$ by $\mathcal{T}[h]=(1+h^2)T[h]$.

\end{lemma}

\begin{proof}[Proof of Lemma \ref{LemT}]

We have that
\begin{equation*}
\begin{split}
T[h]=T[h,h;h]&=-\frac{4}{\pi}\int_{\mathbb{R}}(\Delta_\alpha h)\cdot (\Delta_\alpha h)\cdot \fint_\alpha h\cdot F^2(\fint_\alpha h)d\alpha,
\end{split}
\end{equation*}
and using that
\begin{equation*}
\begin{split}
\vert \Delta_\alpha L_s\vert\lesssim s (\langle\alpha\rangle+\langle y\rangle)^{-1},\qquad
\vert \Delta_\alpha g\vert\lesssim \min\{\alpha^{-\frac{1}{2}},\alpha^{-1}\}\Vert g\Vert_{H^1},
\end{split}
\end{equation*}
we see that
\begin{equation*}
\begin{split}
\sqrt{1+y^2}\vert T[L_s,L_s;g](y)\vert&\lesssim s^2\int \frac{\sqrt{1+y^2}}{1+\alpha^2+y^2}d\alpha\cdot \Vert g\Vert_{L^\infty}\lesssim s^2\Vert g\Vert_{L^\infty},\\
\sqrt{1+y^2}\vert T[L_s,g;h](y)\vert&\lesssim s\int \frac{\sqrt{1+y^2}}{\sqrt{1+\alpha^2+y^2}}\min\{\alpha^{-\frac{1}{2}},\alpha^{-1}\}d\alpha\cdot \Vert g\Vert_{H^1}\Vert h\Vert_{L^\infty}\lesssim s\Vert g\Vert_{H^1}\Vert h\Vert_{L^\infty}\ln\langle y\rangle.
\end{split}
\end{equation*}
and we deduce \eqref{BoundsTLs}. For \eqref{BoundsTLipsitz} it suffices to also show that
\begin{equation*}
\begin{split}
\Vert \int_{\mathbb{R}}\vert (\Delta_\alpha g_1)\cdot (\Delta_\alpha g_2)\vert d\alpha\Vert_{L^\frac{4}{3}\cap L^2}&\lesssim \Vert g_1\Vert_{H^1}\Vert g_2\Vert_{H^1},
\end{split}
\end{equation*}
but this follows from
\begin{equation*}
\begin{split}
\vert \Delta_\alpha g(y)\vert&\le\frac{1}{\vert\alpha\vert}\int_{\{\vert t\vert\le\vert\alpha\vert\}}\vert g^\prime(y+t)\vert dt,\qquad\Vert \Delta_\alpha g\Vert_{L^2_y}\lesssim \vert \alpha\vert ^{-1}\Vert g\Vert_{L^2},
\end{split}
\end{equation*}
which gives
\begin{equation*}
\begin{split}
\Vert (\Delta_\alpha g_1)(\Delta_\alpha g_2)\Vert_{L^1_y\cap L^2_y}&\lesssim \min\{\alpha^{-\frac{1}{2}},\alpha^{-\frac{3}{2}}\}\Vert g_1\Vert_{H^1}\Vert g_2\Vert_{H^1},\\
\end{split}
\end{equation*}
and the proof is complete.
%

\end{proof}

\subsection{Proof of the main estimates for the fixed-point formulation}

\subsubsection{Proof of Lemma \ref{SSLemPi}}

Starting from
\begin{equation*}
\begin{split}
(\frac{2}{\pi}\arctan(z))^2-1=(1-\frac{2}{\pi}\arctan(\frac{1}{z}))^2-1=-\frac{4}{\pi}\arctan(\frac{1}{z})+(\frac{2}{\pi}\arctan(\frac{1}{z}))^2,
\end{split}
\end{equation*}
we see that, for $z=(1+s^2/3)y\ge 1$,
\begin{equation*}
\begin{split}
\pi_s&=-s^3\frac{4}{\pi}\arctan(z)\cdot\frac{2}{\pi}\arctan(\frac{1}{z})+s^3(\frac{2}{\pi}\arctan(\frac{1}{z}))^2\cdot\frac{2}{\pi}\arctan(z),
\end{split}
\end{equation*}
and in particular, $\pi_s$ is $C^\infty$ and
\begin{equation*}
\pi_s(y)=O_{y\to0}(y),\qquad \vert \partial_y^a\pi_s(y)\vert=O_{x\to\infty}(\langle y\rangle^{-1-a}).
\end{equation*}
In particular, we observe that
\begin{equation*}
\begin{split}
\Vert \pi_s\Vert_{H^1}\lesssim s^3.
\end{split}
\end{equation*}
In addition, using \eqref{WLs}, using \eqref{RLs} and \eqref{BoundsTLs}  and direct computations, we see that
\begin{equation*}
\begin{split}
\Vert W[L_s]\partial_yL_s\Vert_{L^1\cap H^1}+\Vert R[L_s]\Vert_{L^\frac{4}{3}\cap L^2}+\Vert \mathcal{T}[L_s]\Vert_{L^\frac{4}{3}\cap L^2}&\lesssim s^3,
\end{split}
\end{equation*}
and using Lemma \ref{TOperators}, we see that
\begin{equation*}
\begin{split}
\Vert \widehat{{\bf T}}_{-1}(W[L_s]\partial_yL_s)\Vert_{H^1}+\Vert \widehat{{\bf T}}_{-1}R[L_s]\Vert_{H^1}+\Vert \widehat{{\bf T}}_{-1}\mathcal{T}[L_s]\Vert_{H^1}&\lesssim s^3.
\end{split}
\end{equation*}
\qed

\subsubsection{Proof of Lemma \ref{SSLemA}}

The proof follows by Neumann series. Using Lemma \ref{TOperators}, we see that
\begin{equation*}
\Vert (L_s^2-s^2)g\Vert_{H^1}+\Vert \widehat{{\bf T}}_{-1}(\vert\nabla\vert(L_s^2-s^2/3)g)\Vert_{H^1}\lesssim s^2\Vert g\Vert_{H^1}.
\end{equation*}
In addition, using Lemma \ref{TOperators}, and Lemma \ref{lemW}, we see that
\begin{equation*}
\begin{split}
\Vert \widehat{{\bf T}}_0[W[L_s]g\Vert_{H^1}&\lesssim \Vert W[L_s]\Vert_{L^\infty}\Vert g\Vert_{H^1}+\Vert \partial_yW[L_s]\Vert_{L^2}\Vert g\Vert_{L^\infty},\\
\Vert \widehat{{\bf T}}_{-1}[g\partial_yW[L_s]]\Vert_{H^1}&\lesssim \Vert g\Vert_{L^2\cap L^\infty}\Vert \partial_yW[L_s]\Vert_{L^2},\\
\Vert \widehat{{\bf T}}_{-1}[W_1[g]\partial_yL_s]\Vert_{H^1}&\lesssim \Vert W_1[g]\Vert_{L^\infty}\Vert \partial_yL_s\Vert_{L^1\cap L^2}.
\end{split}
\end{equation*}
Finally, using Lemma \ref{LemR} and Lemma \ref{LemT}, we see that
\begin{equation*}
\begin{split}
\Vert \widehat{{\bf T}}_{-1}[R_1[g]]\Vert_{H^1}+\Vert \widehat{{\bf T}}_{-1}[\mathcal{T}_1[g]]\Vert_{H^1}&\lesssim s^2\Vert g\Vert_{H^1}.
\end{split}
\end{equation*}
This gives the estimate of $A+Id$.

Next, we compute that
\begin{equation*}
\begin{split}
\partial_sA[g]&=-\left[12s/(3+s^2)^2+\kappa\partial_s(L_s^2-s^2)\right]g+\kappa^2\widehat{\bf T}_{-1}[\vert\nabla\vert (4s/3+\partial_s(L_s^2-s^2))g]-\kappa\widehat{\bf T}_0[g\partial_sW[L_s]]\\
&\quad+\kappa\widehat{\bf T}_{-1}[g\partial_s\partial_yW[L_s]-W_1[g]\partial_y\partial_sL_s-\partial_yL_s\partial_sW_1[g]+\partial_sR_1[g]]+\kappa\widehat{\bf T}_{-1}[\partial_s\mathcal{T}_1[g]]
\end{split}
\end{equation*}
and using that
\begin{equation*}
\begin{split}
\partial_sL_s=s^{-1}L_s+r,\qquad\Vert r\Vert_{H^1}\lesssim s^2,
\end{split}
\end{equation*}
direct calculations and adaptations of Lemma \ref{lemW}, Lemma \ref{LemR} and Lemma \ref{LemT} give the bound on $\partial_sA$.

\qed

\subsubsection{Proof of Lemma \ref{SSLemN}}

Using the fact that $H^1$ is an algebra, we see that
\begin{equation*}
\begin{split}
\Vert L_sg^2\Vert_{H^1}+\Vert g^3\Vert_{H^1}+\Vert \widehat{{\bf T}}_{-1}\left[\vert\nabla\vert(L_sg^2+g^3/3)\right]\Vert_{H^1}&\lesssim \Vert L_s\Vert_{L^\infty}\Vert g\Vert_{H^1}^2+\Vert g\Vert_{H^1}^3,
\end{split}
\end{equation*}
and since these expressions are multilinear they extend to differences. In addition, using Lemma \ref{lemW}, we see that
\begin{equation*}
\begin{split}
\Vert (W_{\ge2}[g_2]-W_{\ge2}[g_1])\p_yL_s\Vert_{L^1\cap L^2}&\lesssim \Vert W_{\ge2}[g_2]-W_{\ge2}[g_1]\Vert_{L^\infty}\Vert \p_yL_s\Vert_{L^1\cap L^2}\\
&\lesssim s\left[s^2+\Vert g_1\Vert_{H^1}^2+\Vert g_2\Vert_{H^1}^2\right]\Vert g_2-g_1\Vert_{H^1},\\
\Vert W_{\ge1}[g_1]\partial_y(g_2-g_1)\Vert_{L^1\cap L^2}&\lesssim \Vert W_{\ge1}[g_1]\Vert_{L^2\cap L^\infty}\Vert \partial_y(g_2-g_1)\Vert_{L^2},\\
\Vert (W_{\ge1}[g_2]-W_{\ge1}[g_1])\partial_yg_2\Vert_{L^1\cap L^2}&\lesssim \Vert W_{\ge1}[g_2]-W_{\ge1}[g_1]\Vert_{L^2\cap L^\infty}\Vert \partial_yg_2\Vert_{L^2}.
\end{split}
\end{equation*}
Similarly, using Lemma \ref{LemR} and Lemma \ref{LemT},
\begin{equation*}
\begin{split}
\Vert R_{\ge2}[g_2]-R_{\ge2}[g_1]\Vert_{L^1\cap L^2}+\Vert \mathcal{T}_{\ge2}[g_2]-\mathcal{T}_{\ge2}[g_1]\Vert_{L^1\cap L^2}&\lesssim \left[s^2+\Vert g_1\Vert_{H^1}^2+\Vert g_2\Vert_{H^1}^2\right]\Vert g_2-g_1\Vert_{H^1},
\end{split}
\end{equation*}
and the proof is complete.

\qed




\section{Numerical Results}
\label{SectionNumerics}

In this section, we describe how to numerically compute  the branch of solutions $k_{s}$ starting from the zero solution. Their existence for small $s$ was proved in Theorem \ref{SSTheorem}. See Figures \ref{fig_discrepancy} and \ref{fig_solutions_normalized} below for different depictions of the solutions.

The main advantage of working with the formulation of equation \eqref{SSMuskat}, as opposed to a self-similar equation for the function $f$, is that all the quantities involved are bounded. Nonetheless there are a few technicalities which we outline below.

The first step consists in changing variables and transforming the infinite domain into a finite one. We do so by setting $y = \tan(z)$ and $\tilde{k}(z) := k(\tan(z)) = k(y)$ so that the domain of definition is mapped into $\left[-\frac{\pi}{2},\frac{\pi}{2}\right]$. This change of variables has been  used successfully  for other problems in fluid mechanics (see for example \cite{Lushnikov-Silantyev-Siegel:collapse-blowup-degregorio} and references therein).

Moreover, we exploit the symmetry to gain an extra cancellation at $\frac{\pi}{2}$. After performing the change of variables and a lengthy calculation, \eqref{SSMuskat} is transformed into 

\begin{align}
0& =\sin(z)\cos(z)\tilde{k}'(z) + \frac{1}{\pi}\int_{0}^{\frac{\pi}{2}}\frac{\tilde{k}'(z)\cos(z)^2 - \tilde{k}'(y)\cos(y)^2}{\tan(z)-\tan(y)}\frac{\sec(y)^2}{1+(\tilde{\Delta}_y \tilde{k}(z))^2} dy \nonumber \\
& -\frac{2}{\pi}\int_{0}^{\frac{\pi}{2}}\left(\frac{\tilde{k}(z)-\tilde{k}(y)}{\tan(z)-\tan(y)}\right)^2\frac{\tilde{\Delta}_y \tilde{k}(z)}{(1+(\tilde{\Delta}_y \tilde{k}(z))^2)^2}\sec(y)^2 dy \nonumber \\
& + \frac{1}{\pi}\int_{0}^{\frac{\pi}{2}}\frac{\tilde{k}'(z)\cos(z)^2 - \tilde{k}'(y)\cos(y)^2}{\tan(z)+\tan(y)}\frac{1}{1+(\tilde{\Delta}_{-y} \tilde{k}(z))^2}\sec(y)^2 dy \nonumber \\
& -\frac{2}{\pi}\int_{0}^{\frac{\pi}{2}}\left(\frac{\tilde{k}(z)+\tilde{k}(y)}{\tan(z)+\tan(y)}\right)^2\frac{\tilde{\Delta}_{-y} \tilde{k}(z)}{(1+(\tilde{\Delta}_{-y} \tilde{k}(z))^2)^2}\sec(y)^2 dy,
\label{SSMuskatCoordinateChange}
\end{align}
where
\begin{align}
\label{DeltaTilde}
\tilde{\Delta}_{y} \tilde{k}(z) = \int_{z}^{y} \frac{\sec(w)^2\tilde{k}(w)}{\tan(y)-\tan(z)}dw.
\end{align}
We performed continuation in $s$ in increments of $\Delta s = 0.1$ and did $70$ iterations, using as initial guess $\tilde{k} = \Delta s \frac{2}{\pi}z$ (the linear approximation) at the first iteration, and for the subsequent ones the result of the previous iteration plus $\Delta s \frac{2}{\pi}z$ to ensure that the updated boundary condition at $z=\frac{\pi}{2}$ is satisfied. In the range we computed, we did not see any impediment towards advancing in $s$, other than computation time, although the errors become bigger as $s$ grows and the algorithm may take one or two more iterations to converge for $s \sim 5$ than for $s \sim 0.1$.

To compute a solution for a fixed $s$, we used the Levenberg-Marquardt algorithm \cite{Levenberg:Levenberg-Marquardt,Marquardt:Levenberg-Marquardt}. 
Our discretization variables consist on the values of $\tilde{k}$ at gridpoints $z_i = \frac{\pi}{2(N-1)}i$, where we are using that $\tilde{k}$ is odd to solve for positive $z$ only. Other strategies such as non-uniform meshes (concentrating points towards 0) or $s$-dependent compactifications would perhaps improve the performance for large $s$ since $\tilde{k}'(0)$ grows with $s$ (see Figure \ref{fig_discrepancy}), but we did not explore them here. We took $N = 129$ and the discrete system we solved was equation \eqref{SSMuskatCoordinateChange} evaluated at $z_i$, $i= 1,\ldots,N-1$ plus the boundary conditions $\tilde{k}(0) = 0, \tilde{k}\left(\frac{\pi}{2}\right) = s$. In order to compute the derivatives we calculated a spline of degree 4 interpolating through the discrete grid and approximated the derivatives of $\tilde{k}$ by the derivatives of the spline. To perform the integration, we integrated in \eqref{SSMuskatCoordinateChange} in the variable $y$ using trapezoidal integration and a grid of $10N = 1290$ points. We also tried finer grids and saw virtually no difference with respect to the results. In order to get stable results, we took care of splitting the domain into 3 regions $(z \sim \frac{\pi}{2}, z \sim y$ and the rest$)$ and computing separately the integrand in each of them, carefully computing the limit. For example, note that despite being bounded at $z \sim \frac{\pi}{2}$, there is a strong instability coming from the fact that the integrand is of the form $\infty - \infty$ if not dealt with properly. The inner integrals in \eqref{SSMuskatCoordinateChange} (i.e. the $\tilde{\Delta}$ integrals) were computed using an adaptive Gauss-Kronrod quadrature of 15 points, also taking care of the limits at $\frac{\pi}{2}$ and at $y$.

\begin{figure}[h!]
\begin{center}
\includegraphics[scale=0.2]{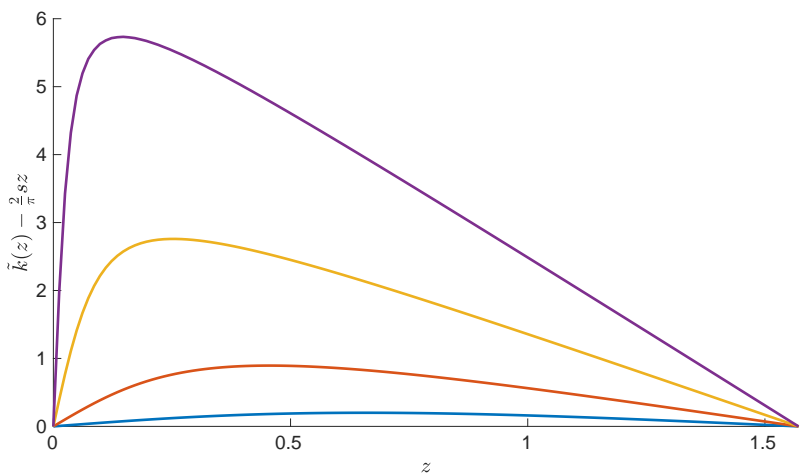}
\caption{\label{fig_discrepancy} Plot of the difference between the numerically computed solutions $\tilde{k}(z)$ and the linear approximation $\frac{2}{\pi}sz = \frac{2}{\pi}s\tan(y)$ for $s=1,2,4,7$. The branch continues beyond what is calculated.}
\end{center}
\end{figure}

\begin{figure}[h!]
\begin{center}
\includegraphics[scale=0.2]{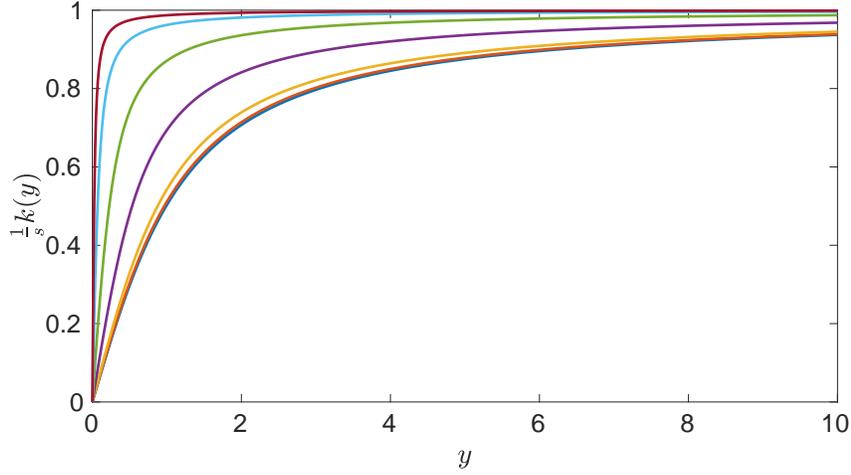}
\caption{\label{fig_solutions_normalized} Comparison of the numerically computed solutions $k(y)$, $y \in [0,10]$ for $s=0.1,0.2,0.4,1,2,4,7$. All functions are normalized to value 1 at infinity. The branch continues beyond what is calculated. The curve corresponding to $s=0.1$ is the lowest one and monotonically increase with $s$.}
\end{center}
\end{figure}

\begin{figure}[h!]
\hspace*{-1.2cm}\begin{tabular}{cc}
 \includegraphics[width=0.45\textwidth]{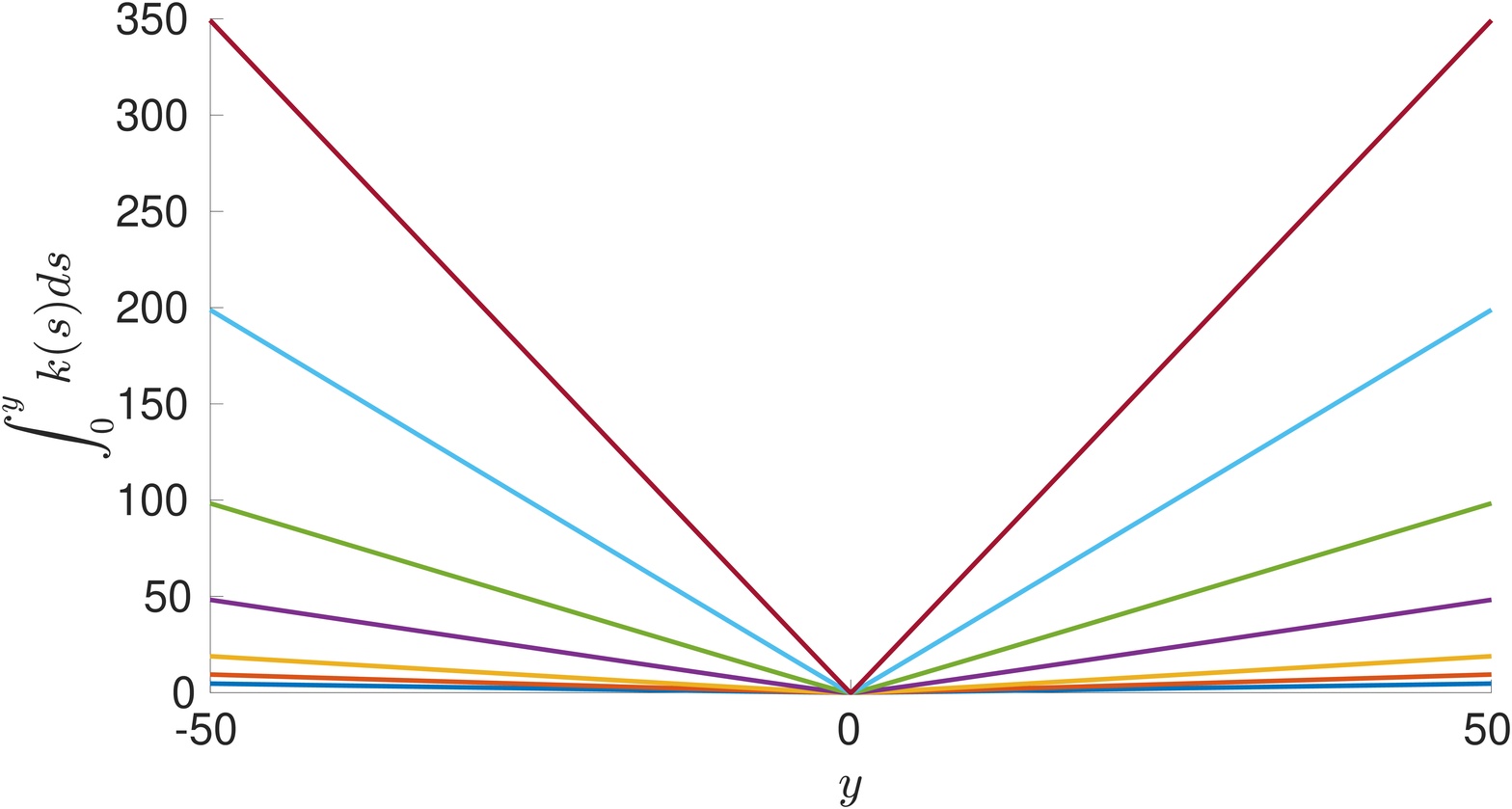}&
 \includegraphics[width=0.45\textwidth]{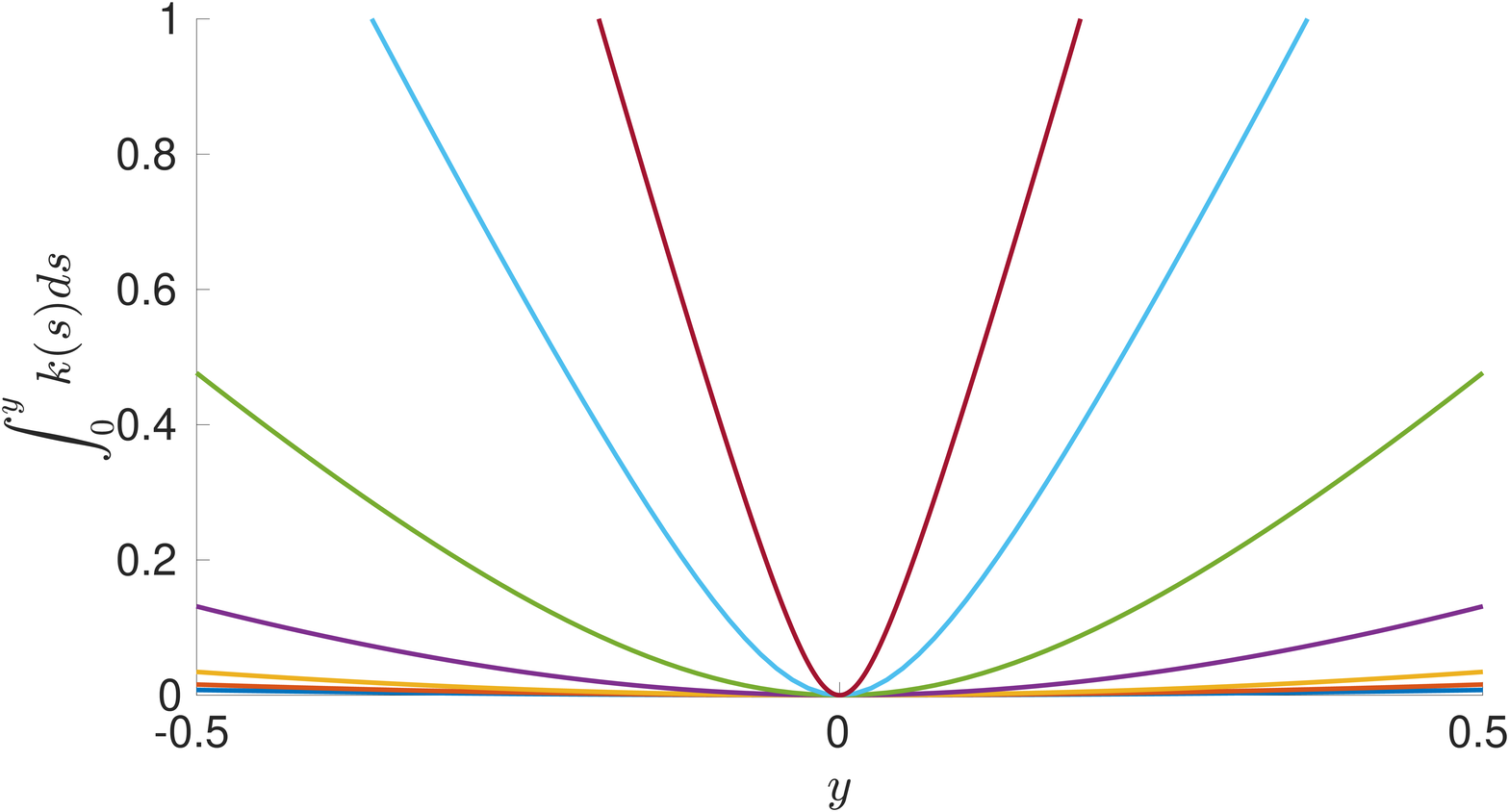} 
\end{tabular}
 \caption{Comparison of the numerically computed solutions $\int_{0}^{y} k(s)ds$,  for $s=0.1,0.2,0.4,1,2,4,7$. Panel (a): $y \in [-50,50]$, panel (b): Close-up, $y \in [-0.5,0.5]$. The curves increase with $s$.
} \label{fig_inteh}
\end{figure}

\subsection*{{\bf Acknowledgments}}
 EGJ and JGS were partially supported by the ERC Starting Grant ERC-StG-CAPA-852741. EGJ was partially supported by the ERC Starting Grant project H2020-EU.1.1.-639227. 
 HQN was partially supported by NSF grant DMS-1907776. BP was supported by NSF grant DMS-1700282 and the CY-Advanced Studies fellow program. We thank Princeton University for computing facilities (Polar Cluster).
 This project has received funding from the European Union’s Horizon 2020 research and innovation programme under the Marie Sklodowska-Curie grant agreement CAMINFLOW No 101031111.
 
\vspace{0.3cm}

\bibliographystyle{plain}
\bibliography{references}

\vspace{1cm}

\end{document}